\documentclass[12pt]{amsart}
\usepackage{amsmath,amssymb,latexsym, amsthm, amscd, mathrsfs, stmaryrd, mathtools}
\usepackage[linktocpage=true]{hyperref}
\usepackage{cleveref}
\usepackage{dynkin-diagrams}
\usepackage{xcolor}
\usepackage{graphicx}
\usepackage[all]{xy}
\usepackage{enumitem}
\usepackage{chngcntr}
\usepackage{nicematrix}
\usepackage{multirow}
\usepackage{array}
\usepackage{booktabs}
\usepackage{tikz}
\usepackage{tikz-cd}
\usetikzlibrary{arrows.meta}
\usepackage{bbm}
\usepackage{longtable}
\usepackage{hyperref}
\usepackage[sorting=nyt,style=alphabetic,backend=bibtex,hyperref=true,doi=false,maxbibnames=9,maxcitenames=4,url=false,maxcitenames=50,maxbibnames=50]{biblatex}
\crefformat{section}{\S#2#1#3} 
\crefformat{subsection}{\S#2#1#3}
\crefformat{subsubsection}{\S#2#1#3}
\renewbibmacro{in:}{}
\newtheorem{thm}{Theorem} [subsection]
\newtheorem{thm2}{Theorem} [section]
\theoremstyle{definition}

\newcommand{\bal}{{\mbox{\boldmath$\alpha$}}}
\newtheorem{rem}[thm]{Remark}
\newtheorem{rem2}[thm2]{Remark}
\theoremstyle{plain}
\newtheorem{prop}[thm]{Proposition}

\newtheorem{lem}[thm]{Lemma}
\newtheorem{cor}[thm]{Corollary}
\newtheorem{conj}[thm]{Conjecture}

\numberwithin{equation}{section}

\newcommand{\C}{\mathbb C}

\newcommand{\esi}{\mathfrak{e}_6}
\newcommand{\ese}{\mathfrak{e}_7}
\newcommand{\ee}{\mathfrak{e}_8}
\newcommand{\ff}{\mathfrak{f}_4}

\newcommand{\g}{\mathfrak{g}}

\newcommand{\gl}{\mathfrak{gl}}
\newcommand{\h}{\mathfrak{h}}

\newcommand{\KK}{\mathbb K}

\newcommand{\ov}{\overline}

\newcommand{\sll}{\mathfrak{sl}_2}

\newcommand{\Z}{\mathbb Z}

\makeatletter
\newcommand{\extp}{\@ifnextchar^\@extp{\@extp^{\,}}}
\def\@extp^#1{\mathop{\bigwedge\nolimits^{\!#1}}}
\makeatother

\makeatletter
\newcommand*{\@rowstyle}{}
\newcommand*{\rowstyle}[1]{
  \gdef\@rowstyle{#1}%
  \@rowstyle\ignorespaces%
}
\newcolumntype{=}{
  >{\gdef\@rowstyle{}}%
}
\newcolumntype{+}{
  >{\@rowstyle}%
}
\makeatother

\DeclareMathOperator*{\ad}{ad}
\DeclareMathOperator*{\sdim}{sdim}
\DeclareMathOperator*{\im}{im}
\DeclareMathOperator*{\Rep}{Rep}
\DeclareMathOperator*{\Vecc}{Vec}
\DeclareMathOperator*{\sVec}{sVec}
\DeclareMathOperator{\Ver}{Ver}

\pgfdeclarearrow{
name = jibladze,
parameters = { \the\pgfarrowlength },
setup code = {},
drawing code = {
  \pgfsetdash{}{0pt} 
      \pgfsetroundjoin   
  \pgfsetroundcap    
  \pgfsetlinewidth{5\pgflinewidth}
  \pgfsetstrokecolor{white}
  \pgfpathmoveto{\pgfpoint{-.75\pgfarrowlength}{\pgfarrowlength}}
  \pgfpathlineto{\pgfpoint{0}{0}}
  \pgfpathlineto{\pgfpoint{-.75\pgfarrowlength}{-\pgfarrowlength}}
  \pgfusepathqstroke
  \pgfsetlinewidth{.2\pgflinewidth}
  \pgfsetstrokecolor{black}
  \pgfpathmoveto{\pgfpoint{-.75\pgfarrowlength}{\pgfarrowlength}}
  \pgfpathlineto{\pgfpoint{0}{0}}
  \pgfpathlineto{\pgfpoint{-.75\pgfarrowlength}{-\pgfarrowlength}}
  \pgfusepathqstroke
},
defaults = { length = 2pt }
}

\title[New Constructions of Exceptional Simple Lie Superalgebras]{New Constructions of Exceptional Simple Lie Superalgebras With Integer Cartan Matrix in Characteristics $3$ and $5$ Via Tensor Categories}

\author[A.S. Kannan]{Arun S. Kannan}%
\address{Department of Mathematics, Massachusetts Institute of Technology, Cambridge, MA 02139} \email{akannan@mit.edu}

\begin{document}

\begin{abstract}
Using tensor categories, we present new constructions of several of the exceptional simple Lie superalgebras with integer Cartan matrix in characteristic $p = 3$ and $p = 5$ from the complete classification of modular Lie superalgebras with indecomposable Cartan matrix and their simple subquotients over algebraically closed fields by Bouarroudj, Grozman, and Leites in 2009. Specifically, let $\bal_p$ denote the kernel of the Frobenius endomorphism on the additive group scheme $\mathbb{G}_a$ over an algebraically closed field of characteristic $p$. The Verlinde category $\Ver_p$ is the \textit{semisimplification} of the representation category $\Rep \bal_p$, and $\Ver_p$ contains the category of super vector spaces as a full subcategory. Each exceptional Lie superalgebra we construct is realized as the image of an exceptional Lie algebra equipped with a nilpotent derivation of order at most $p$ under the semisimplification functor from $\Rep \bal_p$ to $\Ver_p$.
\end{abstract}

\maketitle

\setcounter{tocdepth}{1}
\tableofcontents

\section{Introduction}
In \cite{bouarroudj2009classification}, the finite-dimensional modular Lie superalgebras over $\KK$ with indecomposable Cartan matrix are classified, where $\KK$ is an algebraically closed field. In characteristic $5$, the classification includes Lie superalgebras arising from reducing the non-exceptional classical (or serial) simple Lie superalgebras over $\mathbb{C}$ modulo $5$, reductions of the parametric family $\mathfrak{osp}(4|2; \alpha)$ modulo $5$ for $\alpha \in \Z_{\neq 0}$, the Brown Lie superalgebra $\mathfrak{brj}_{2;5}$, and the Elduque Lie superalgebra $\mathfrak{el}(5; 5)$. In characteristic $3$, the classification includes the mod $3$ reductions of the non-exceptional classical simple Lie superalgebras, reductions of the parametric family $\mathfrak{osp}(4|2; \alpha)$ modulo $3$ for $\alpha \in \Z_{\neq 0}$, the Brown Lie superalgebra $\mathfrak{brj}_{2;3}$, the ten Elduque and Cunha Lie superalgebras, and a characteristic $3$ analog $\mathfrak{el}(5; 3)$ of $\mathfrak{el}(5; 5)$.
\par
In this paper, we use the semisimplification of tensor categories to produce new constructions of the ten Elduque and Cunha Lie superalgebras in characteristic $3$, the Elduque Lie superalgebra in characteristic $5$ first discovered in \cite{elduque2007some}, its characteristic $3$ analog discovered in \cite{bouarroudj2009classification}, and the Brown Lie superalgebra in characteristic $3$. In particular, we can consider a Lie algebra $\g$ over $\KK$ with a nilpotent derivation $d$ of order at most $p$, where $p > 2$ is the characteristic of $\KK$; this can then be realized as a Lie algebra in the category $\Rep \KK[t]/(t^p)$ of $\KK[t]/(t^p)$-modules by specializing $t$ to $d$. The semisimplification of this category is the Verlinde category $\Ver_p$, which contains as a full subcategory the category of super vector spaces $\sVec_\KK$. Therefore, the image $\ov{\g}$ of $\g$ under the semisimplification functor projected onto this full subcategory is a Lie algebra in $\sVec_\KK$, which is precisely a Lie superalgebra. For more details, one can refer to \cite{etingof2018koszul, etingof2019semisimplification, ostrik2015symmetric}.
\par
In every case, we start with one of the exceptional simple Lie algebras $\mathfrak{br}_3, \ff, \esi, \ese$, and $\ee$ and then choose an appropriate nilpotent element of order $3$ or $5$, and then semisimplify. The software SuperLie (cf. \cite{SuperLie}) greatly simplified verifying this process. Most of these constructions are a consequence of the main theorem, Theorem \ref{maint}. Specifically, in characteristic $3$, one can start with a finite-dimensional Lie algebra $\g(A)$ with Cartan matrix $A$ and choose a suitable nilpotent element $e \in \g(A)$ that is the sum of various Chevalley generators to realize $\g(A)$ as an object in $\Rep \bal_3$. Then, the semisimplification of the derived algebra $\g(A)^{(1)}$ is the derived algebra $\g(\widetilde{A})^{(1)}$, where $\widetilde{A}$ is some other Cartan matrix related to $A$ and can be determined by the choice of $e$. Either $\g(\widetilde{A})$ is simple, or $\g(\widetilde{A})^{(1)}/\mathfrak{c}$ is simple, where $\mathfrak{c}$ is the center of the derived algebra. This actually gives us semisimplifications of many finite-dimensional Lie algebras, not just the exceptional ones. For characteristic $3$, the results are summarized in the following table:

\begin{longtable}[c]{@{}lll@{}}
  \toprule
   Lie algebra & Nilpotent element  & Lie superalgebra \\* \midrule
  \endfirsthead
  \multicolumn{3}{c}%
  {{\bfseries Table continued from previous page}} \\
  \toprule
   Lie algebra & Nilpotent element(s)  & Lie superalgebra  \\* \midrule
  \endhead
  \bottomrule
  \endfoot
  \endlastfoot
  $\mathfrak{br}_3$ & $e_1, e_2$ & $\mathfrak{brj}_{2;3}$  \\ \midrule
   $\ff$ &  $e_4$ & $\g(1,6)$  \\ \midrule
   $\esi^{(1)}$ & $e_1, e_2, e_6$ &  $\g(2, 6)^{(1)}$ \\
   & $e_1 + e_2$, $e_2 + e_6$, $e_1 + e_6$  & $\g(3,3)^{(1)}$  \\
   & $e_1 + e_2 + e_6$  & $\g(2,3)^{(1)}$ \\ \midrule
   $\ese$ & $e_1, e_2, e_7$ & $\g(4,6)$  \\
   & $e_1 + e_2$, $e_2 + e_7$, $e_1 + e_7$ &  $\mathfrak{el}(5; 3)$ \\
   & $e_1 + e_2 + e_7$  & $\g(4,3)$ \\
   & $e_1 + e_2 + e_5 + e_7$  & $\g(1, 6)$ \\ \midrule
   $\ee$ & $e_1, e_2, e_8$ &  $\g(8, 6)$  \\
   & $e_1 + e_2$, $e_2 + e_8$, $e_1 + e_8$  & $\g(6, 6)$ \\
   & $e_1 + e_2 + e_8$  &  $\g(8, 3)$  \\
   & $e_1 + e_2 + e_6 + e_8$ & $\g(3, 6)$  \\* \bottomrule
  \end{longtable}
Finally, in characteristic $5$, semisimplifying $\ee$ with respect to $e_2 + e_3 + e_4$ gives the Elduque Lie superalgebra $\mathfrak{el}(5;5)$.
\par
The original construction of the Elduque and Cunha Lie superalgebras are based on symmetric composition algebras and the Elduque Supermagic Square, an analog of Freudenthal's Magic Square, where division algebras are used to construct the exceptional Lie superalgebras (cf. \cite{elduque2006new,cunha2006extended,cunha2007extended}). A conceptual explanation relating this original approach to semisimplification is developed in \cite{elduque2022}. In \cite{bouarroudj2006cartan,bouarroudj2020roots}, explicit descriptions are given in terms of Cartan matrices and generators and relations. It can be difficult to work with such descriptions of the Lie superalgebras, and the author hopes that the theoretical framework of symmetric tensor categories makes working with these Lie superalgebras more tractable. 
\par
It would be interesting to see if this approach can be applied to construct vectorial simple Lie superalgebras, which were not considered, or the remaining exceptional simple Lie superalgebras with integer Cartan matrix we were not able to construct. Specifically, there are four such Lie superalgebras. Three of these come from the parametric family of Lie superalgebras $\mathfrak{osp}(4|2; \alpha)$, which are also known as $D(2|1; \alpha)$ in the literature: when $\alpha = 2$ in characteristic $p = 3$ and when $\alpha = 2,3$ in characteristic $p = 5$. These are deformations of $\mathfrak{osp}(4|2) = \mathfrak{osp}(4|2; 1)$. The fourth is the Brown Lie superalgebra $\mathfrak{brj}_{2;5}$ in characteristic $5$. We pose the following \textbf{open problem}: construct the remaining Lie superalgebras using semisimplification, or explain why it is not possible. In the case that it is possible, this may help us find an analog to Theorem \ref{maint} in characteristic $5$, which in turn may help us find a statement for all characteristics. 
\par
Lastly, one can also use semisimplification to construct representations of Lie superalgebras by semisimplifying representations of the Lie algebras they come from; very little is known about the representation theory of the exceptional Lie superalgebras, so this may be particularly interesting. 
\par 
\textbf{Acknowledgements.} The author would like to deeply thank his advisor, Pavel Etingof, for pointing out that these constructions were possible and for his patience and guidance. The author would also like to thank Julia Plavnik,  Guillermo Sanmarco, and Iv\'{a}n Angiono for pointing out that Theorem \ref{maint} can be applied to construct the Brown Lie superalgebra $\mathfrak{brj}_{2;3}$. The author would finally like to thank the two anonymous referees for their invaluable feedback. This paper is based upon work supported by The National Science Foundation Graduate Research Fellowship Program under Grant No. 1842490 awarded to the author.

\textbf{Disclosures.} The author declares that he has no conflict of interest.

\section{Contragredient Lie Superalgebras}\label{contraglie}
The Lie superalgebras listed above can all be realized as contragredient Lie superalgebras, which means they arise from a Cartan matrix. \footnote{There are other, inconsistent usages of the term ``contragredient" in the literature, but for convenience we will use it rather than saying ``Lie superalgebra with Cartan matrix" every time.} We will use this formulation to prove that each Lie superalgebra we construct from semisimplification is isomorphic to the corresponding exceptional simple Lie superalgebra. For the reader's convenience, we recall the basics on contragredient Lie superalgebras here.
\par
Fix an algebraically closed field $\KK$ of characteristic $p \geq 0$. Let $n$ be a positive integer, and let $A = (a_{ij})$ be an arbitrary $n\times n$ matrix with entries in $\Z$, called the \textit{Cartan matrix}. Let $\overline{A} = (\ov{a}_{ij})$ be the matrix obtained by reducing the entries of $A$ modulo $p$ (if $p = 0$, $A = \ov{A}$), and let $s$ be the rank of $\ov{A}$. Let $\widetilde{\mathfrak{h}}$ be a vector space over $\KK$ of dimension $2n - s$. Let $\{h_i\}_{i=1}^{n}$ be a collection of $n$ linearly independent vectors in $\widetilde{\mathfrak{h}}$, and let $\{\alpha_i\}_{i=1}^n$ be vectors in $\widetilde{\mathfrak{h}}^*$ such that $\alpha_j(h_i) = \ov{a}_{ij}$ for all $1 \leq i, j \leq n$, where we interpret $\Z/p\Z \subset \KK$ in the usual way. Let $I = \{i_1, \dots, i_n\} \subset (\Z/2\Z)^n$ be a collection of parities, and let $\mathrm{par}(j) \coloneqq i_j$ for $1 \leq j \leq n$.
\par
Then, we define the Lie superalgebra $\widetilde{\g}(A, I)$ as follows: it is the Lie superalgebra generated by $e_1, \dots, e_n, f_1,\dots, f_n$, called \textit{Chevalley generators}, and $\widetilde{\mathfrak{h}}$ such that the parity of $e_j$ and $f_j$ is $i_j \in I$ for all $1 \leq j \leq n$, $\widetilde{\h}$ is purely even, and such that they are subject to the following relations:

\begin{equation}\label{rels}
  [e_i, f_j] = \delta_{ij}h_i; \ \ [h, e_j] = \alpha_j(h)e_j; \ \ [h, f_j] = -\alpha_j(h)f_j; \ \ [\widetilde{\mathfrak{h}}, \widetilde{\mathfrak{h}}] = 0,
\end{equation}
for all $1 \leq i, j \leq n$ and $h \in \widetilde{\mathfrak{h}}$. The Lie algebra $\widetilde{\h}$ is the \textit{maximal torus}.
\par
It can be shown that $\widetilde{\g}(A, I)$ admits a triangular decomposition of Lie subsuperalgebras $\widetilde{\g}(A, I) = \widetilde{\mathfrak{n}}^- \oplus \widetilde{\mathfrak{h}} \oplus \widetilde{\mathfrak{n}}^+$, where $\widetilde{\mathfrak{n}}^-$ is the  Lie subsuperalgebra freely generated by the $f_i$'s and $\widetilde{\mathfrak{n}}^+$ is the Lie subsuperalgebra freely generated by the $e_i$'s. As usual, there is a root space decomposition, which gives a $Q$-grading, where

\begin{equation}\label{qgrading}
  Q \coloneqq \bigoplus_{k=1}^n \Z\alpha_k; \ \ \mathrm{deg}(e_i) = \alpha_i; \ \ \mathrm{deg}(f_i) = -\alpha_i; \ \ \mathrm{deg}(h) = 0
\end{equation}
for $1 \leq i \leq n$ and $h \in \widetilde{\mathfrak{h}}$.
\par
We define the contragredient Lie superalgebra $\mathfrak{g}(A, I)$ to be the quotient $\widetilde{\g}(A, I)/\mathfrak{r}$, where $\mathfrak{r}$ is the maximal graded ideal that trivially intersects $\widetilde{\mathfrak{h}}$ (indeed, such an ideal is unique because the sum of any two such ideals also trivially intersects $\widetilde{\mathfrak{h}}$). The quotienting preserves the triangular decomposition in the sense that $\mathfrak{r} = \mathfrak{r}\cap \widetilde{\mathfrak{n}}^- \oplus  \mathfrak{r}\cap \widetilde{\mathfrak{n}}^+$. Hence, the triangular decomposition descends to a triangular decomposition on 
\[\mathfrak{g}(A, I) = \mathfrak{n}^- \oplus \widetilde{\mathfrak{h}} \oplus \mathfrak{n}^+.\] Similarly, we have a root space decomposition and a $Q$-grading on $\mathfrak{g}(A, I)$. The images of $\{e_i, f_i, h_i\}$ for $1 \leq i \leq n$ under the projection map by $\mathfrak{r}$ are still linearly independent, so we will by abuse of notation still use $\{e_i, f_i, h_i\}$ to refer to these images and also call them Chevalley generators. The context will make it clear which set of Chevalley generators we will refer to.
\par
The derived algebra $\widetilde{\g}(A, I)^{(1)}$ of $\widetilde{\g}(A, I)$ will be more useful for us because it is generated by the Chevalley generators (but these two are the same when $\ov{A}$ has full rank). It is easily seen that 

\[\widetilde{\g}(A, I)^{(1)} = \widetilde{\mathfrak{n}}^- \oplus \mathfrak{h} \oplus \widetilde{\mathfrak{n}}^+,\] 
where $\mathfrak{h}$, called the \textit{Cartan subalgebra}, is the subspace of $\widetilde{\mathfrak{h}}$ spanned by $h_1, \dots, h_n$. \footnote{This terminology is not standard. The usual definition of the Cartan subalgebra is the maximal nilpotent subalgebra coinciding with its normalizer. However, it is convenient to use this term to refer to $\mathfrak{h}$.} The maximal graded ideal of $\widetilde{\g}(A, I)^{(1)}$ that trivially intersects $\mathfrak{h}$ is the same as $\mathfrak{r}$, so it is also easily seen that 

\[\widetilde{\g}(A, I)^{(1)}/\mathfrak{r} \cong \g(A, I)^{(1)} = \mathfrak{n}^- \oplus \mathfrak{h} \oplus \mathfrak{n}^+.\] 
The Chevalley generators clearly generate $\g(A, I)^{(1)}$. By an abuse of language, we will also refer to subquotients of $\g(A,I)$ as contragredient Lie superalgebras, as otherwise it can be extremely inconvenient to be explicit. We will also refer to $A$ as the Cartan matrix of these subquotients. For instance, we will call $\g(A, I)^{(1)}$ a contragredient Lie superalgebra, and sometimes it has a center $\mathfrak{c}$; in this case, we will also call $\g(A, I)^{(1)}/\mathfrak{c}$ a contragredient Lie superalgebra.
\begin{rem2}
  The difference between $\mathfrak{g}(A, I)$ and $\mathfrak{g}(A, I)^{(1)}$ is that the former has a maximal torus $\widetilde{\mathfrak{h}}$ which is potentially larger than $\mathfrak{h}$, which happens when $\ov{A}$ doesn't have full rank. We will work with the latter because it is generated by the Chevalley generators. A reason for working with $\mathfrak{g}(A, I)$ is because the maximal torus $\widetilde{\h}$ enables one to do representation theory, in the sense that one can construct a non-degenerate invariant bilinear form, a generalized Casimir operator, and highest weight modules (cf. \cite{Kac1990InfiniteDim}). The situation here is totally analogous to affine Lie algebras, where one extends the central extension $\g[t, t^{-1}] \oplus \KK c$ of a loop algebra by the derivation $d = t\tfrac{d}{dt}$ (cf. \cite{kac1987bombay}).
\end{rem2}
\par
If one takes $I$ to consist of solely zero parities, then the Lie superalgebra is purely even and we recover the definition of a contragredient Lie algebra (cf. \cite{Kac1990InfiniteDim}). We point out that unlike the classical case, ``inequivalent" choices of Cartan matrices $A$ and parity set $I$ can produce isomorphic Lie superalgebras (cf. \cite{bouarroudj2009classification} for a definition of ``inequivalent''). We can also do a similar construction where the entries of $A$ are arbitrary elements of $\KK$ (cf. \cite{cunha2007extended}), but the definitions provided will suffice for our purposes and are more convenient.
\par
From now on, we will assume that any Cartan matrix $A = (a_{ij})$ will satisfy the following properties: \footnote{These properties are a variation of those of what is usually called \textit{a generalized Cartan matrix} (cf. \cite{hoyt2007classification}).}
\begin{enumerate}
  \item $a_{ii} \in \{\ov{0}, 2\}$ if $\mathrm{par}(i) = 0$;
  \item $a_{ii} \in \{0,1\}$ if $\mathrm{par}(i) = 1$;
  \item if $a_{ii} \in \{1,2\}$ then $a_{ij} \in \Z_{\leq 0}$ for all $j$;
  \item $a_{ij} = 0 \Leftrightarrow a_{ji} = 0$ for all $i \neq j$.
\end{enumerate}
We use the notation $a_{ii} = \ov{0}$ to indicate that $a_{ii} = 0$ but $i$ is even. Under these assumptions, we can deduce the parity set $I$ from $A$ (and hence, it may be omitted and we will write $\g(A)$ instead of $\g(A, I)$). We emphasize that we consider a very special class of Cartan matrices. 
\par
A Cartan matrix $A$ is \textit{indecomposable} if there is no suitable permutation of rows and columns of $A$ that can make $A$ block diagonal. Lastly, we will say $A$ is \textit{symmetrizable} if it factors as the product of a diagonal matrix and a symmetric matrix. In general, one can take elements in the ground field (cf. \cite{hoyt2007classification,chapovalov2010classification} for definitions over $\mathbb{C}$ and \cite{bouarroudj2009classification,bouarroudj2020roots} in the modular case).
\par
Recall that a \textit{Kac-Moody Lie algebra} over $\KK$ is a contragredient Lie algebra with symmetrizable Cartan matrix (with respect to the purely even parity set). For Kac-Moody Lie algebras over $\mathbb{C}$ with Cartan matrix $A = (a_{ij})$, which include the simple Lie algebras, the defining relations of the ideal $\mathfrak{r}$ are the well known \textit{Serre relations} (cf. \cite{Kac1990InfiniteDim}):
\begin{align*}
    (\mathrm{ad} \ {e_i})^{1-a_{ij}}(e_j) = (\mathrm{ad} \ {f_i})^{1-a_{ij}}(f_j) = 0
\end{align*}
for all $1 \leq i \neq j \leq n$. Over $\KK$ these relations also hold, but there may be others. In general, the defining relations are not easily determined. For all of the exceptional simple Lie superalgebras we consider, there exists an analog of the Serre relations (cf. Proposition  2.5.1 in \cite{bouarroudj2006cartan}). In fact, a defining set of relations in terms of Chevalley generators has been determined (cf. \cite{bouarroudj2010divided,bouarroudj2008new}).
\par
For more details on the theory of contragredient Lie (super)algebras and double extensions of loop (super)algebras, one can refer to \cite{Kac1990InfiniteDim, serganova2011kac} and the last two sections in \cite{bls2019}.

\section{Constructing Lie Superalgebras from Lie Algebras Using Semisimplification}
In this section, we describe how one can construct a Lie superalgebra from a Lie algebra using \textit{semisimplification}, which will be the procedure used to construct the exceptional simple Lie superalgebras. From now on, we will assume that $\KK$ is an algebraically closed field of characteristic $p > 2$. This section will draw from the theory of \textit{symmetric tensor categories}, which abstract the key properties of the representation category of an affine (super) group scheme. A down-to-earth reference is \cite{etingof2021lectures}; a more thorough treatment is given in \cite{etingof2016tensor}.
\subsection{Operadic Lie Algebras}
Recall that a Lie algebra over $\KK$ is a vector space $\g$ endowed with a $\KK$-bilinear map $\beta: \g \times \g \rightarrow \g$ which is anti-symmetric (assuming $\mathrm{char} \ \KK \neq 2$) and satisfies the Jacobi identity. This can be phrased categorically as follows. The category $\Vecc_\KK$ of vector spaces over $\KK$ is a symmetric tensor category endowed with the usual braiding $c_{X, Y}: X \otimes Y \rightarrow Y \otimes X$ given by interchanging $X$ and $Y$, a natural isomorphism in objects $X$ and $Y$. Then, a Lie algebra (in the category $\Vecc_\KK$) is an object $\g$ equipped with a morphism $\beta: \g\otimes \g \rightarrow \g$ such that the following relations of morphisms hold:

\begin{align*}
    \beta \circ (1_{\g \otimes \g} + c_{\g, \g}) = 0, \\
    \beta \circ (\beta \otimes 1_\g) \circ (1_{\g^{\otimes 3}} + (123)_{\g^{ \otimes 3}} + (132)_{\g^{\otimes 3}}) = 0,
\end{align*}
where the permutation $(123)_{\g^{ \otimes 3}}: \g^{\otimes 3} \rightarrow \g^{\otimes 3}$ is given by \[(123)_{\g^{ \otimes 3}} \coloneqq (1_{\g} \otimes c_{\g \otimes \g}) \circ (c_{\g \otimes \g}\otimes 1_{\g}),\] 
and the permutation $(132)_{\g^{\otimes 3}}: \g^{\otimes 3} \rightarrow \g^{\otimes 3}$ is given by 

\[(132)_{\g^{\otimes 3}}\coloneqq (c_{\g \otimes \g}\otimes 1_{\g}) \circ (1_{\g} \otimes c_{\g \otimes \g}).\] 
Here, we ignore the associativity morphisms. The first relation corresponds to the anti-symmetry condition, and the second is the Jacobi identity. Using these as defining axioms, we can extend the definition to any symmetric tensor category $\mathcal{C}$ with braiding $c$, and call the pair $(\g, \beta)$ an operadic Lie algebra in $\mathcal{C}$. We can also allow $\g$ to be an ind-object in the category. For the category we will consider, ind-objects will be the potentially-infinite direct sums of simple objects in the category.
\par
Recall that the category of super vector spaces $\sVec_\KK$ consists of $\Z/2\Z$-graded vector spaces and morphisms. In particular, we write a super vector space $V$ as  $V = V_{\ov{0}}\oplus V_{\ov{1}}$, and let $\sdim V = (\dim V_{\ov{0}} | \dim V_{\ov{1}})$. Here $\ov{0},\ov{1} \in \Z/2\Z$ and distinguish the even and odd subspaces respectively.
This category has a braiding $c_{X, Y}$ given by the Koszul sign rule:
\begin{equation}\label{svecrule}
  c_{X,Y}(x\otimes y) = (-1)^{|x||y|}(y \otimes x),
\end{equation}
where $x, y$ are homogeneous (i.e. purely even or purely odd). We call a Lie algebra in the category of super vector spaces $\sVec_\KK$ a \textit{Lie superalgebra}. 
\begin{rem}\label{liesupalgrem}
  One should note that in characteristic $3$, the usual definition of a Lie superalgebra and the definition given above do not coincide, as for any odd element $x$, the relation $[x, [x, x]] = 0$ required for a Lie superalgebra does not follow from the Jacobi identity. Without this relation imposed, one has the notion of a \textit{weak Lie superalgebra}, and the ideal generated by this relation is actually just the linear span of elements of the form $[x,[x,x]]$ for odd $x$. However, this will not be a major concern for our considerations.
\end{rem}
In a symmetric tensor category in general we do not have the notion of elements or vectors in an object. However, for the purposes of this paper, we will be working with Lie algebras in the representation category of a certain finite group scheme, so we can refer to vectors in objects of this category by applying the forgetful functor into $\Vecc_\KK$. Furthermore, using vectors will make it easier to talk about and describe the bracket.
\par
For more details on operadic Lie algebras, one can consult \cite{etingof2018koszul}.

\subsection{The Verlinde Category}\label{Verp}
Our goal is to now describe a symmetric tensor category whose semisimplification contains the category of super vector spaces. This way, we can start with a Lie algebra in that symmetric tensor category and semisimplify to get a Lie algebra in the semisimplification; projecting onto $\sVec_\KK$ will give us a Lie superalgebra.
\par
Let $\bal_p$ denote the kernel of the Frobenius endomorphism on the additive group scheme $\mathbb{G}_a$ over an algebraically closed field $\KK$ of characteristic $p > 0$. Its coordinate ring $\KK\bal_p$ is $\KK[t]/(t^p)$, which is a cocommutative Hopf algebra with comultiplication defined by letting $t$ be primitive (this only works in characteristic $p$). The dual space $\KK\bal_p^*$ of $\KK\bal_p$ has basis given by $\{f_0, f_1, \dots, f_{p-1}\}$, where $f_i(t^k) = \delta_{ik} i!$. The comultiplication on $\KK\bal_p$ gives a multiplication on $\KK\bal_p^*$ where $f_0$ is the identity and $f_{i}f_{j} = f_{i+j}$ (let $f_{i} = 0$ for $i \geq p$). Therefore, as algebras, $\KK\bal_p$ and $\KK\bal_p^*$ are isomorphic under the map $t^i \mapsto f_i$. Because modules over the affine group scheme $\bal_p$ are determined by $\KK\bal_p$-comodules, which themselves are $\KK\bal_p^*$-modules, we will describe objects in the representation category $\Rep \bal_p$ of $\bal_p$ as finite-dimensional $\KK[t]/(t^p)$-modules. For the remainder of this text, the symbol $t$ will be used to refer to the corresponding element of $\KK[t]/(t^p)$.
\par
The category $\Rep \bal_p$ is a symmetric tensor category with braiding given by the usual braiding of vector spaces (there is a forgetful functor from $\Rep \bal_p$ to $\Vecc_\KK)$. Hence, an example of a Lie algebra $(\g, \beta)$ in $\Rep \bal_p$ is a Lie algebra in $\Vecc_\KK$ equipped with a nilpotent element $x \in \g$ of order at most $p$; then $\g$ is a $\KK\bal_p$-module by letting $t$ act as $\mathrm{ad} \ x$, and $\beta$ is naturally a morphism in $\Rep \bal_p$ by the Jacobi identity (as a Lie algebra in $\Vecc_\KK$). More generally, we can take $t$ to be any nilpotent derivation of order at most $p$ (not necessarily inner).
\par
The category $\Rep \bal_p$ is not semisimple; indeed, it contains non-simple indecomposable objects. The pairwise non-isomorphic indecomposable objects are given by the modules $J_n = \KK^n$ where $t$ acts as the nilpotent Jordan block of size $n$ ($1 \leq n \leq p$). If $v_1, v_2, \dots, v_n$ is a basis of $J_n$ such that $t\cdot v_i = v_{i+1}$, we will use the notation

\[v_1 \mapsto v_2 \mapsto \cdots \mapsto v_n\]
to refer to that particular object $J_n$.
\par
The semisimplification of this category is by definition the Verlinde category $\Ver_p$. Formally speaking, this is the symmetric tensor category obtained by quotienting out by the tensor ideal of \textit{negligible morphisms}, which are morphisms $f: V \rightarrow W$ such that for all morphisms $g: W \rightarrow V$, the trace of the composition $f \circ g$ is zero. Intuitively, the effect of this is forcing Schur's lemma to hold. In other words, the semisimplification of a symmetric tensor category is the symmetric tensor category obtained by declaring all indecomposable objects to be simple, except those whose categorical dimension is zero, which are sent to zero. We then define the tensor product the same way (for more details on semisimplification, see \cite{etingof2019semisimplification}). The semisimplification is a semisimple symmetric tensor category by construction. There is a semisimplification functor from a symmetric tensor category $\mathcal{C}$ to its semisimplification $\ov{\mathcal{C}}$, and it is symmetric and monoidal. We will denote the images of objects under this functor with an overline over the original object. 
\par
Therefore, the simple objects in $\Ver_p$ are $L_1, \dots, L_{p-1}$, which are the images of $J_1, \dots, J_{p-1}$ under the semisimplification functor, respectively, i.e. $L_i = \ov{J_i}$. If $v_1 \mapsto v_2 \mapsto \cdots \mapsto v_i$ denotes a $J_i$, we will refer to the corresponding copy of $L_i$ by $\ov{v_1\mapsto v_2 \mapsto \cdots \mapsto v_i}$ (for $i < p$). Note that $J_p$ is sent to the zero object as it is $p$-dimensional, so its categorical dimension is $0$. In terms of negligible morphisms, this is because any sequence of morphisms $J_i \rightarrow J_p \rightarrow J_i$ and $J_p \rightarrow J_i \rightarrow J_p$ for any $i$ has trace zero, so in the semisimplification there are no nonzero morphisms in or out of the image of $J_p$, meaning its image is zero. It is well known that the tensor product is given by the truncated Clebsch-Gordan rule (cf. \cite{ostrik2015symmetric}), which is similar to the usual Clebsch-Gordan rule of $\mathfrak{sl}_2(\C)$-modules (the truncation comes from the terms in bold):

\begin{equation}\label{clebschgordan}
  L_m \otimes L_n = \bigoplus_{i=1}^{\min(m,n,\boldsymbol{p-m,p-n})} L_{|m-n| + 2i - 1}.
\end{equation}
In particular, $\mathbbm{1} \coloneqq L_1$ is the unit object with respect to tensor product. More importantly, we have the following proposition:

\begin{prop}\label{svecc}
The category $\sVec_\KK$ is symmetric tensor equivalent to the subcategory generated by the objects $L_1$ and $L_{p-1}$ in $\Ver_p$.
\end{prop}
\begin{proof} This follows from results in \cite{ostrik2015symmetric}, but we offer a direct, linear-algebraic proof. We have $L_{p-1} \otimes L_{p-1} = L_1$ by the truncated Clebsch-Gordan rule \eqref{clebschgordan}. Therefore, one needs to check that the induced braiding on $\Ver_p$ from $\Rep \bal_p$ under the semisimplification functor  restricts to this subcategory appropriately. Let us use $c$ to denote both the braiding and its image.
\par
In $\Rep \bal_p$,  let $J_{p-1}$ be given by the basis $v_1 \mapsto v_2 \mapsto \dots \mapsto v_{p-1}$. It is known that in $\Rep \bal_p$, 

\[J_{p-1} \otimes J_{p-1} = J_1 \oplus (p-2)J_p\]
 (cf. \cite{green1962modular}). One such decomposition is as follows. For each $1 \leq i \leq p-2$, a copy of $J_{p}$ arises from the submodule generated by the vector $v_{i}\otimes v_{1}$. It is clear that $ t^p\cdot(v_{i}\otimes v_1) = 0$ but $t^{p-1}\cdot(v_{i} \otimes v_1)$ is nonzero, so these do indeed give a copy of $J_p$. Since $t^p = 0$, these copies of $J_p$ are direct summands in $J_{p-1} \otimes J_{p-1}$. The copy of $J_1$ arises as the span of the vector $\sum_{i=1}^{p-1} (-1)^{i} v_i \otimes v_{p-i}$. To see that this copy of $J_1$ is the last direct summand, let us define the degree of $v_{i} \otimes v_{j}$ to be $i+j$. This gives a grading on $J_{p-1} \otimes J_{p-1}$, where the action of $t$ increases the grading by $1$, and the copies of $J_p$ and $J_1$ above are graded submodules. Suppose the copy of $J_1$ intersected the span of the copies of $J_p$, meaning it lies in the graded subspace spanned by $t^{p-1}\cdot(v_{i} \otimes v_1)$ for $1 \leq i \leq p-2$. The vector $\sum_{i=1}^{p-1} (-1)^{i} v_i \otimes v_{p-i}$ has degree $p$, but homogeneous vectors in this graded subspace all have degree greater than or equal to $p+1$. From this contradiction, we deduce that $\sum_{i=1}^{p-1} (-1)^{i} v_i \otimes v_{p-i}$ spans a direct summand.
\par
The copies of $J_p$ vanish in the semisimplification, so let's just look at the copy of $J_1$ spanned by the vector $\sum_{i=1}^{p-1} (-1)^{i} v_i \otimes v_{p-i}$. Since $p$ is odd, this sum has an even number of terms, and the braiding map 

\[c_{J_{p-1}, J_{p-1}}: J_{p-1}\otimes J_{p-1} \rightarrow J_{p-1}\otimes J_{p-1}\] 
will multiply this element by $-1$. Therefore, in the semisimplification, we have $L_{p-1} \otimes L_{p-1} = L_1$, and the image of the braiding $c_{J_{p-1}, J_{p-1}}$ under the semisimplification functor is

\[c_{L_{p-1}, L_{p-1}}: L_{p-1}\otimes L_{p-1} \rightarrow L_{p-1}\otimes L_{p-1}\] 
and given by multiplication by $-1$. Then, it is clear that image of $c$ under the semisimplification functor restricted to the subcategory generated by $L_1$ and $L_{p-1}$ behaves as described in equation \eqref{svecrule}.
\end{proof}
From now on we will refer to the full subcategory generated by $L_1$ and $L_{p-1}$ as $\sVec_\KK$.  This means that the sum of the isotypic components of $L_1$ and $L_{p-1}$ of any object $X$ in $\Ver_p$ is a super vector space, and we will refer to this subobject of $X$ as the \textit{projection} of $X$ onto $\sVec_\KK$. Similarly, we can project morphisms onto $\sVec_\KK$. By functoriality, if we are given a Lie algebra $(\mathfrak{g}, \beta)$ in $\Rep \bal_p$, the projection $(\ov\g, \ov{\beta})$ of its semisimplification in $\Ver_p$ onto $\sVec_\KK$ is a Lie algebra in $\sVec_\KK$, which is a Lie superalgebra. Therefore, over fields of characteristic $p$, we can produce Lie superalgebras from Lie algebras by specifying a nilpotent element of order at most $p$.

\begin{rem}
  As noted in Remark \ref{liesupalgrem}, to get a Lie superalgebra in characteristic $3$, one must quotient out by the further relation $[x, [x,x]] = 0$ for odd $x$. For instance, consider the free Lie algebra $\g$ on $x,y$ modulo the elements of degree $4$; this is a Lie algebra with basis $\{x, y, [x, y], [x,[x,y]], [y,[y,x]]\}$. Define the derivation $d$ by $d(x) = y$ and $d(y) = 0$. Then, $\g$ can be realized as an object in $\Rep \bal_3$ with respect to $d$, where we have a copy of $J_1$ for $[x,y]$ and copies of $J_2$ given by $x \mapsto y$ and $[x,[x,y]] \mapsto -[y,[y,x]]$. Semisimplifying gives an operadic Lie superalgebra with odd generator $z$ and basis $\{z, [z,z], [z,[z,z]]\}$. The even part is $1$-dimensional and the odd part is $2$-dimensional.
\end{rem}

\begin{rem}
The setup above explains how we construct Lie superalgebras by semisimplifying a Lie algebra in $\Rep \bal_p$. However, the above discussion excludes characteristic $2$, as $\Ver_2$ is just $\Vecc_\KK$. There is a procedure by which one can take a simple Lie algebra and produce a simple Lie superalgebra, but this is a completely different approach from semisimplification. This procedure is described in \cite{leitesLieSupalgchar2} and reduces the classification simple Lie superalgebras in characteristic $2$ to that of simple Lie algebras.
\end{rem}

\subsection{Semisimplifications of Non-Exceptional Classical Lie Algebras}\label{ssclassical}
In this section, we discuss semisimplifications of non-exceptional  classical (or otherwise known as serial) Lie algebras. These semisimplifications are previously known. For instance, computing them is left as an exercise to the reader in Chapter $9$ of \cite{etingof2016tensor}. 
\par
Now, if $V$ is an object in $\Rep \bal_p$, then $\ov{\gl(V)} = \gl(\ov{V})$. This is because the semisimplification functor is a symmetric monoidal functor that preserves duals, meaning 

\[\ov{\gl(V)} = \ov{V \otimes V^*} = \ov{V}\otimes \ov{V}^* = \gl(\ov{V}).\]
Furthermore, the bracket $\beta: \gl(V) \otimes \gl(V) \rightarrow \gl(V)$ given by

\[\beta = 1_V \otimes ev_{V^*,V} \otimes 1_{V^*} \circ (1_{\gl(V)\otimes \gl(V)} - c_{\gl(V), \gl(V)})\] semisimplifies to the bracket desired:

\[\ov{\beta} = 1_{\ov{V}} \otimes ev_{\ov{V}^*,\ov{V}} \otimes 1_{\ov{V}^*} \circ (1_{\gl(\ov{V})\otimes \gl(\ov{V})} - c_{\gl(V), \gl(V)}),\]
where $ev_{W^*,W}: W^* \otimes W \rightarrow \KK$ denotes the evaluation morphism on an object $W$ by its dual $W^*$. In particular, if we start with $V = mJ_1 \oplus nJ_{p-1} \oplus lJ_p$, then $\gl(V) = \gl_{m + n(p-1) + lp}$, and the semisimplification is $\gl(\ov{V}) = \gl_{m|n}$.
\par
When $p > 2$, a similar statement holds for the symplectic and orthogonal Lie algebras, which is the consequence of a more general setup.  Suppose $V$ is an object in $\Rep \bal_p$ and $\gamma$ is a non-degenerate bilinear form on $V$ in $\Rep \bal_p$. In particular, this means we can view $\gamma$ as an isomorphism $V \rightarrow V^*$ in $\Rep \bal_p$. This gives us a map $\phi: \gl(V) \rightarrow \gl(V)$ given by 

\[\phi = (1_{V}\otimes \gamma^{-1}) \circ (1_{V\otimes V} + c_{V \otimes V})\circ (1_{V}\otimes \gamma).\] 
Since the outside morphisms are isomorphisms and the middle morphism is twice a projector, the kernel of $\phi$ is a direct summand of $\gl(V)$. The kernel of the middle map is $\bigwedge^2(V)$, and we can identify this kernel with the kernel of $\phi$. On the other hand, the kernel of $\phi$ is by definition the Lie algebra that preserves the form, so via the form $\gamma$, we can say $\bigwedge^2(V)$ is the Lie subalgebra of $\gl(V)$ that preserves $\gamma$.
\par
When we semisimplify, $\ov{\gamma}$ is a non-degenerate bilinear form on $\ov{V}$, and $\ov{\bigwedge^2(V)} = \bigwedge^2{(\ov{V})}$ (because it is the degree $2$ piece of the exterior algebra; for degrees greater than or equal to the characteristic, this may not necessarily be true). We deduce that semisimplification of the Lie algebra in $\Rep \bal_p$ preserving a form is the Lie algebra in $\Ver_p$ that preserves the semisimplification of the form.
\par
For an explicit construction of $\mathfrak{osp}_{m|2n}$, we can start with an $m$-dimensional vector space $V_0$ on which $\KK[t]/(t^p)$ acts trivially. Fix a non-degenerate symmetric bilinear form $B_0$ on $V_0$. Because the $t$-action on $V_0$ is trivial, it is immediate that $B_0: V_0 \otimes V_0 \rightarrow \mathbb{K}$ is a morphism in $\Rep \bal_p$. Then, let $V_1$ be a $2n(p-1)$-dimensional vector space with ordered basis 

\[\{w_1^1, w_2^1, \dots, w_{p-1}^1\} \cup \cdots \cup \{w_{1}^{2n}, w_{2}^{2n}, \dots, w_{p-1}^{2n}\}.\] 
We realize $V_1 = 2nJ_{p-1}$ as an object in $\Rep \bal_p$ by $t\cdot w_i^j = w_{i+1}^j$ and $t\cdot w_{p-1}^j = 0$  for $1 \leq i \leq p-2$ and  $1 \leq j \leq 2n$. Define a non-degenerate symmetric bilinear form $B_1: V_1 \otimes V_1 \rightarrow \KK$ with respect to the given basis by the $n \times n$ block-diagonal matrix

\[
 B_1 = 
 \begin{bmatrix}
  R & & \\
  & R & & \\
  & & \ddots & \\
  & & & R
 \end{bmatrix},
\]
where $R$ is the following block matrix:

\[ 
R = 
\begin{bmatrix}
  0 & S \\
  -S & 0
\end{bmatrix}.
\]
Here, $S$ is a $(p-1) \times (p-1)$ anti-diagonal matrix with alternating entries $1, -1, 1, \dots, -1$:

\[
S = 
\begin{bmatrix}
  & & & & -1 \\
  & & & 1 & \\
  & & \iddots & & \\
  & -1 & & & \\
  1 & & & &
\end{bmatrix}.
\]
Notice that $p-1$ is necessarily even and so $S^T = -S$.  It follows that $B_1$ is a non-degenerate symmetric matrix, and it is defined this way so that $B_1: V_1 \otimes V_1 \rightarrow \KK$ is a morphism in $\Rep \bal_p$. Then, the Lie algebra in $\Rep \bal_p$ preserving the non-degenerate symmetric bilinear form $B = B_0 \oplus B_1$ on $V = V_0 \oplus V_1$ is $\mathfrak{o}_{m + 2n(p-1)}$. Its semisimplification in $\Ver_p$ is $\mathfrak{osp}_{m|2n}$.
\par
Similarly, we can get $\mathfrak{osp}_{n|2m}$ by semisimplifying $\mathfrak{sp}_{2m + n(p-1)}$. We have a similar setup except we make the following modifications. Now take $V_0$ to be $2m$-dimensional (still with the trivial action of $t$) with non-degenerate alternating form $B_0$. Furthermore, $V_1 = nJ_{p-1}$ is $n(p-1)$-dimensional with ordered basis as before, except we only take the first $n(p-1)$ basis vectors. With respect to this basis we can define a non-degenerate alternating form $B_1: V_1 \otimes V_1 \rightarrow \KK$ given by $B_1 = diag(S, S, \dots, S)$ (there are $n$ blocks on the diagonal). Then, the Lie algebra in $\Rep \bal_p$ preserving the non-degenerate alternating bilinear form $B = B_0 \oplus B_1$ on $V = V_0 \oplus V_1$ is $\mathfrak{sp}_{2m + n(p-1)}$. Its semisimplification in $\Ver_p$ is $\mathfrak{osp}_{n|2m}$.
\par
In the preceding discussion about orthosymplectic Lie superalgebras, $B_1$ could have been any symmetric (in the first scenario) or alternating (in the second scenario) non-degenerate bilinear form on $V_1$ in $\Rep \bal_p$, but the point is to show that such a form exists.

\subsection{Explicit Description of Semisimplification in Characteristic $3$}\label{explicit}
The language of symmetric tensor categories is naturally suited for talking about semisimplification. However, since it is relatively new, we offer an explicit description of what happens to the Lie bracket under semisimplification using linear algebra. We will use this language in our proofs below.
\par
Let $\g$ be a Lie algebra in $\Rep \bal_3$ with respect to some derivation $d$ such that $d^3 = 0$. We can pick a non-canonical decomposition $\g = n_1J_1 \oplus n_2J_2 \oplus n_3J_3$. For each copy of $J_1$, we pick a basis $x_i$, where $1 \leq i \leq n_1$. For each copy of $J_2$, we pick a basis $x_i \mapsto x_i'$, where $1 + n_1 \leq i \leq n_1 + n_2$. Finally, for each copy of $J_3$, we pick a basis $x_i \mapsto x_i' \mapsto x_i''$ for $1 + n_1 + n_2 \leq i \leq n_1 + n_2 + n_3$. The collection $\{x_i,x_j,x_j',x_k,x_k', x_k''\}$ for $1 \leq i \leq n_1$, $1 +n_1 \leq j \leq n_1 + n_2$ , $1+n_2+n_3 \leq k \leq n_1+n_2+n_3$  is a basis of $\g$.
\par
After semisimplification, the $J_3$ terms vanish, and the $J_1$ terms and $J_2$ terms give $L_1$ terms and $L_2$ terms, respectively, which collectively give rise to a basis of the Lie superalgebra $\ov{\g}$. In particular, for $1 \leq i \leq n_1$, the copy of $J_1$ corresponding to $x_i$ gives an even basis vector $y_i$ in $\ov{\g}$ which spans the subspace $\ov{x_i}$. Here, by abuse of notation, $x_i$ refers to the basis vector we picked above and the $J_1$ it spans, so that the notation $\ov{x_1}$ makes sense. For $n_1 + 1 \leq i \leq n_1 + n_2$, the copy of $J_2$ corresponding to $x_i \mapsto x_i'$ gives an odd basis vector $y_i$ in $\ov{\g}$ which spans the subspace $\ov{x_i\mapsto x_i'}$. Later in the text, we will simply write $\ov{x_i}$ or $\ov{x_i \mapsto x_i'}$ in place of $y_i$, even though these are subspaces of $\ov{\g}$ (this is another abuse of notation). For instance, if $n_1 = 2$, then the bracket $[y_1, y_2]$ will be written as $[\ov{x_1}, \ov{x_2}]$ and the bracket $[y_2, y_3]$ will be written as $[\ov{x_2}, \ov{x_3 \mapsto x_3'}]$. Lastly, if $v = \sum_{i=1}^{n_1} a_ix_i$ for suitable $a_i$, then the notation $\ov{v}$ is defined to mean $\sum_{i=1}^{n_1} a_i \ov{x_i}$.
\par
Now, we will describe the structure constants of this basis of $\ov{\g}$.
\begin{prop}\label{strucconsts}
  Let $C_{ij}^k$ denote the structure constants of the basis $\{y_i\}$ of $\ov{\g}$ above, i.e. $[y_i, y_j] = \sum_k C_{ij}^ky_k$. Then,

  \begin{enumerate}
    \item if $i,j,k \leq n_1$, or $1 \leq i \leq n_1$ and $j,k > n_1$, or $j < n_1$ and $i,k > n_1$, then $C_{ij}^k$ is the coefficient of $x_k$ in $[x_i, x_j]$;
    \item if $i,j > n_1$ and $k \leq n_1$, then $C_{ij}^k$ is the coefficient of $x_k$ in $-[x_i, x_j'] + [x_i', x_j]$;
    \item in all other cases, $C_{ij}^k = 0$.
  \end{enumerate}
\end{prop}
\begin{proof}
  Proof of $1)$ and $3)$ follow easily from the definition of $\{y_i\}$. The proof of $2)$ is a consequence of the proof of Proposition \ref{svecc}.
\end{proof}
Let's do this calculation explicitly for $\g = \mathfrak{gl}_3$. Let $e_{ij}$ refer to the elementary matrix with a $1$ in the $(i,j)$ entry and zero elsewhere. Then, with respect to the adjoint action of $e_{23}$, $\g$ is an object in $\Rep \bal_3$. We choose the following decomposition: $e_{11}$, $e_{11} + e_{22} + e_{33}$ as copies of $J_1$; $e_{12} \mapsto -e_{13}$ and $e_{31} \mapsto e_{21}$ as copies of $J_2$; and $e_{32} \mapsto e_{22} - e_{33} \mapsto e_{23}$ as a copy of $J_3$. Then, the basis vectors of $\ov{\g}$ are 

\begin{align*}
  &y_1 = \ov{e_{11}} &y_2 = \ov{e_{11} + e_{22} + e_{33}} \\
  &y_3 = \ov{e_{12} \mapsto -e_{13}} &y_4 = \ov{e_{31} \mapsto e_{21}}.
\end{align*}
\par
Applying the formulas in Proposition \ref{strucconsts}, we have: $[y_1, y_2] = 0$, $[y_1, y_3] = y_3$, and $[y_1, y_4] = -y_4$. We can also compute $[y_3, y_4]$. Both of these vectors are odd, so we look at

\[-[e_{12}, e_{21}] + [-e_{13}, e_{31}] = (e_{22} - e_{11}) - (e_{11} - e_{33}) = e_{11} + e_{22} + e_{33},\]
so $[y_3, y_4] = y_2$. Therefore, the resulting Lie superalgebra is $\gl_{1|1}$. This confirms what we expect from \S\ref{ssclassical}.

\subsection{An Example}
Let us consider a more complicated example. Consider $\g_2$, the $14$-dimensional exceptional simple Lie algebra of rank $2$.  It has Cartan matrix $A = \begin{pmatrix} 2 & -3 \\ -1 & 2\end{pmatrix}$ and Dynkin diagram:

 \[\dynkin[labels={\beta,\alpha},scale=4, arrow shape/.style={-{jibladze[length=7pt]}}, label distance = 5pt]  G{oo}\]
where $\beta$ corresponds to the index $1$ and $\alpha$ to the index $2$. Its root system can be visualized graphically as the following:

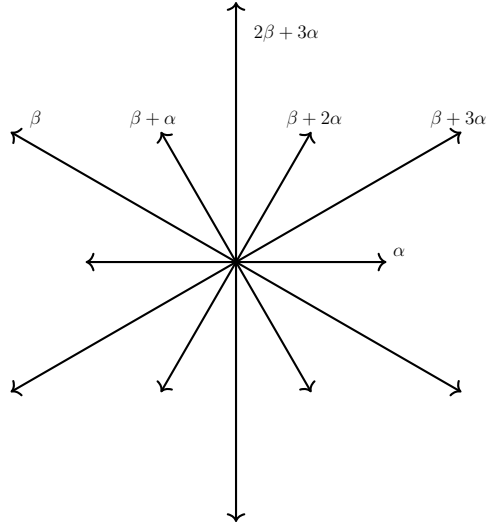
\begin{center}
  \begin{tikzpicture}
    \foreach\ang in {60,120,...,360}{
     \draw[->,black!80!black,thick] (0,0) -- (\ang:2cm);
    }
    \foreach\ang in {30,90,...,330}{
     \draw[->,black!80!black,thick] (0,0) -- (\ang:3.46cm);
    }
    \node[anchor=south west,scale=0.6] at (2,0) {$\alpha$};
    \node[anchor=south east,scale=0.6] at (-2.5,1.7) {$\beta$};
    \node[anchor=south west,scale=0.6] at (-1.5,1.7) {$\beta + \alpha$};
    \node[anchor=south east,scale=0.6] at (1.5,1.7) {$\beta + 2\alpha$};
    \node[anchor=south west,scale=0.6] at (2.5,1.7) {$\beta + 3\alpha$};
    \node[anchor=south west,scale=0.6] at (0.15,2.85) {$2\beta + 3\alpha$};
  \end{tikzpicture}
\end{center}
However, in characteristic $3$, the root spaces that involve $\pm 3\alpha$ form an ideal that trivially intersects the Cartan subalgebra. Therefore, the construction of a Lie algebra $\g(A)$ in terms of its Cartan matrix $A$ tells us that if $A$ is reduced modulo 3, then $\g(A)$ is a $10$-dimensional Lie algebra with basis
 
\[e_1, e_2, [e_1, e_2], [e_2, [e_1, e_2]]\] 
for the upper triangular subalgebra,  

\[f_1, f_2, [f_1, f_2], [f_2, [f_1, f_2]]\] 
for the lower triangular subalgebra, and $h_1, h_2$ for the Cartan subalgebra.
\par
Because $e_3$ is nilpotent of degree $3$, we can realize $\g_2$ as an object in $\Rep \bal_3$ with respect to $e_3$. Let's see what happens when we semisimplify. We have the following decomposition into indecomposables (which is not canonical). The copies of $J_1$ are given by $[e_2, [e_1, e_2]]$, $[f_2, [f_1, f_2]]$, and $h_2$. The copies of $J_2$ are given by $e_2 \mapsto [e_1, e_2]$ and $[f_1,f_2] \mapsto f_2$. A copy of $J_3$ arises from $f_1 \mapsto h_1 \mapsto e_1$. Therefore, the semisimplification of $\g_2$ with respect to $e_1$ is a Lie superalgebra of superdimension $(3|2)$.
\par
Now, let's compute the bracket on this Lie superalgebra using Proposition \ref{strucconsts}. The bracket of $[f_2, [f_1, f_2]]$ and $[e_2, [e_1, e_2]]$ is $h_2$; the adjoint action of $h_2$ on $[f_2, [f_1, f_2]]$ is $2[f_2, [f_1, f_2]]$ and on $[e_2, [e_1, e_2]]$ it is $-2[e_2, [e_1, e_2]]$. Because these correspond to copies $J_1$, their semisimplification and therefore the even part of $\ov{\g_2}$ is $\sll$. To compute the action of this even part on the odd part, we note that $[h_2, [f_1, f_2]] = [f_1, f_2]$ and $[h_2, f_2] = -2f_2 = f_2$. Similarly, $[h_2, e_2] = 2e_2 = -e_2$ and $[h_2, [e_1,e_2]] = -[e_1, e_2]$. This tells us that the weights of the odd part as an $\sll$-module are $\pm 1$. So the odd part is the two-dimensional tautological $\sll$-module. Finally, one can consider the bracket on the odd part. We have 

\[ \left[\ov{[f_1,f_2] \rightarrow f_2}, \ov{[f_1,f_2] \rightarrow f_2} \right] = -\ov{[f_2,[f_1,f_2]]};\]
\[ \left[\ov{[f_1,f_2] \rightarrow f_2},  \ov{e_2 \rightarrow [e_1, e_2]}\right] = \ov{-[[f_1,f_2], [e_1,e_2]] + [f_2,e_2]} = \ov{h_2}.\]
\[ \left[\ov{e_2 \rightarrow [e_1, e_2]}, \ov{e_2 \rightarrow [e_1, e_2]} \right] = \ov{[e_2,[e_1,e_2]]}.\]
Putting this all together, we deduce that the semisimplification $\ov{\g_2}$ of $\g_2$ is the Lie superalgebra $\mathfrak{osp}_{1|2}$.

\subsection{Main Theorem}\label{maintsec}
We are now ready to state the main theorem of this paper. Let us first introduce some notation. Let $\g(A)$ be a finite-dimensional contragredient Lie algebra with Cartan matrix $A = (a_{ij})$ of size $n$. Let $\mathcal{I} = \{i_1, i_2, \dots, i_l\}$ be a subset of boundary nodes of the Dynkin diagram of $\g(A)$ such that each chosen boundary node has one single edge coming out of it, such that the chosen boundary nodes are pairwise non-adjacent, and such that no two chosen boundary nodes share an adjacent node. Let $\mathcal{J} = \{j_1,\dots,j_l\}$ be the indices such that $j_r$ is the node connected to $i_r$ for each $r$ between $1$ and $l$, inclusive. Finally, we require that $\mathcal{I}$ be chosen so that $a_{ii} = 2$ for all $i \in \mathcal{I} \cup \mathcal{J}$. The following picture of the Dynkin diagram of $\ese$ illustrates an example:
\\
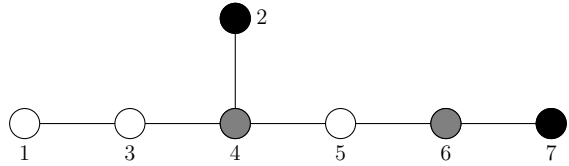
\begin{center}\label{boundarynodes}
  \begin{dynkinDiagram}[mark=o,labels={1,2,3,4,5,6,7},scale=4, label distance = 5pt]E7
  \fill[black, draw=black] (root 2) circle (0.05cm);
  \fill[black, draw=black] (root 7) circle (0.05cm);
  \fill[gray, draw=black] (root 4) circle (0.05cm);
  \fill[gray, draw=black] (root 6) circle (0.05cm);
  \end{dynkinDiagram}
\end{center}
The boundary nodes are $1, 2$ and $7$. The nodes $2$ and $7$, colored in black, are the chosen subset of boundary nodes, and the nodes $4$ and $6$, colored in gray, are the nodes attached to the boundary nodes $2$ and $7$, respectively. Hence, $l = 2$ and $\{i_1, i_2\} = \{2,7\}$ and $\{j_1,j_2\} = \{4,6\}$.
\par
Let $\widetilde{A} = (\widetilde{a}_{ij})$ be the $(n-l) \times (n-l)$ matrix obtained from $A$ by setting $a_{j_r,j_r} = 0$ for all $r$ and deleting the row and column attached to $i_r$ for all $1 \leq r \leq l$. Note that the $e_{i_r}$'s pairwise commute, as do the $f_{i_r}$'s.
\par
Before proceeding, it is useful to review the constructions in \S\ref{char3constructions} and look closely at the Cartan matrices of the exceptional Lie algebra and the exceptional Lie superalgebra it semisimplifies to. The key idea is that we have a copy of $J_2$ given by $e_{j_r} \mapsto [e_{i_r}, e_{j_r}]$ for each $1 \leq r \leq l$, and in the semisimplification these merge to form an odd Chevalley generator. Now, let's state some supporting lemmas.

\begin{lem}\label{obincat}
 The element $e = \sum_{r=1}^{l} e_{i_r}$ is nilpotent of degree $3$, and $\g(A)^{(1)}$ can be realized as a Lie algebra in $\Rep \bal_3$ w.r.t. $\ad e$.
\end{lem}
\begin{proof}
  Without loss of generality, by suitably reordering the indices, we may assume that $i_1 = 1, i_2 = 2, \dots, i_l = l$ and $j_1 = 1 + l, j_2 = 2 + l, \dots, j_l = 2l$. Then, by the Serre relations, we have $[e_i, e_j] = 0$ for all  $1 \leq i \leq l$ and $1 \leq j \neq i + l \leq n$ and $[e_i, e_{i+l}] \neq 0$ but $[e_i, [e_i, e_{i+l}]] = 0$. So each $e_i$ is nilpotent of degree $3$ and $\ad e_i$ pairwise commute for $1 \leq i \leq l$. It follows by the binomial theorem that $\ad e = \ad e_1 + \cdots + \ad e_l$ cubes to zero in characteristic $3$. This shows that $\g(A)$ can be realized as an object in $\Rep \bal_3$.
\end{proof}
For the remainder of this section, we will assume that the indices are reordered as in the proof of Lemma \ref{obincat}. Recall the $Q$-grading on $\g(A)$ in \eqref{qgrading}.

\begin{lem}\label{bas}
  There exists a basis $B$ of $\g(A)^{(1)}$ in which $\ad e$ acts by the direct sum of Jordan blocks, and the generators $e_i, f_i$ ($1 \leq i \leq n$) of $\g(A)^{(1)}$ together with a suitable basis of the Cartan subalgebra $\h$ of $\g(A)^{(1)}$ collectively form a subset of $B$. 
\end{lem}
\begin{proof}
  We will construct such a basis. It is useful to extend the action of $e$ to the action of the $\sll$-triple $\{e, f, h\}$, where $f = f_1 + f_2 + \cdots + f_l$ and $h = [e, f] = h_1 + h_2 + \cdots + h_l$. We have the following blocks of type $J_1$ in $\g(A)^{(1)}$:

  \[e_{1+2l}, \dots, e_{n} \ \mathrm{and} \ f_{1+2l}, \dots, f_{n}.\]
  There are blocks of type $J_2$ formed by:
  \[e_{1+l} \mapsto [e, e_{1+l}], \dots, e_{2l} \mapsto [e, e_{2l}].\]
  Since $a_{i_r,j_r} = -1$ for all $r$, we also have other blocks of type $J_2$:

  \[[f,f_{1+l}] \mapsto f_{1+l}, \dots,[f, f_{2l}] \mapsto f_{2l}.\]
  We then have the following $\sll$-triples giving $J_3$'s:

  \[f_1 \mapsto h_1 \mapsto e_1, \dots, f_l \mapsto h_l \mapsto e_l.\]
  Finally, we can consider the remaining part of the Cartan subalgebra. We have additionally the following blocks of type $J_1$:

  \[h_{1+2l}, \dots, h_{n}, h_{1+l} - h_{1}, \dots, h_{2l} - h_{l}.\]
  We claim that the sum $W$ of all of these copies of $J_i$ is a direct sum. First, note that each $J_i$ is $Q$-graded even though $e$ is not $Q$-homogeneous because $[e, e_{i+l}] = [e_i, e_{i+l}]$, $[f, f_{i+l}] = [f_i, f_{i+l}]$, and $[e_i, e_j] = [f_i, f_j] = 0$ for $1 \leq i,j \leq l$ by the Serre relations. Then, because each root space appearing here is $1$-dimensional and because 
  
  \[\{h_1, \dots, h_l, h_{1+l} - h_1, \dots, h_{l+l} - h_l, h_{2l+1}, \dots, h_n\}\] 
  is a basis of the Cartan subalgebra, their sum is direct.
  \par
  Now, we argue that $W$ is a direct summand of $\g(A)^{(1)}$ in $\Rep \bal_3$. First, note that $W$ is $Q$-graded because each $J_i$ is. This means that the height grading on $\g(A)^{(1)}$ also restricts to $W$. Let $\g_i$ denote the subspace of height $i$ in $\g(A)^{(1)}$. Then, $W$ contains $\g_{-1}, \g_0, \g_1$ and has trivial intersection with $\g_i$ for $|i| > 2$. Let $W'$ be the span of the root spaces that do not intersect $W$. The subspace $W'$ is also $Q$-graded and graded by height, has trivial intersection with each of $\g_{-1}, \g_0, \g_1$, and contains $\g_i$ for $|i| > 2$. Clearly, $\g(A)^{(1)} = W \oplus W'$ as vector spaces; we need to show that $W'$ is $\ad e$-invariant.
  \par
  To do so, consider any $x \in W'$, and let $x = \sum_{|i| \geq 2} x_i$ be a height decomposition with $x_i \in \g_i$. Now, note that $e$ raises height by $1$ on homogeneous vectors, so for $i \leq -4$ and $i \geq 2$, we have $[x, x_i] \in W'$. Therefore, it suffices to assume that $x$ is of the form $x = x_{-3} + x_{-2}$.
  \par
  The terms $[e, x_{-2}]$ and $[e, x_{-3}]$ have different heights, so we can look at them individually. Let's start with $[e, x_{-2}]$. This has height $-1$, so it lies in $W$. Hence, we can write $[e, x_{-2}] = \sum_{i=1}^n c_i f_i$ for suitable $c_i$. Suppose $[e, x_{-2}] \neq 0$. By the $Q$-grading and because $e = \sum_{k=1}^l e_k$, this means we can write 
  
  \[x_{-2} = \sum_{k=1}^l\sum_{i=1}^n c_{ki}[f_k, f_i].\] 
  However, by the Serre relations, $[f_k, f_i] = 0$ for all $1 \leq k \leq l$ and $1 \leq  i \neq k + l \leq n$. Hence, $x_{-2} = \sum_{k=1}^l c_{k,k+l} [f_k, f_{k+l}] $ which lies in $W$, a contradiction if $x_{-2} \neq 0$. Therefore, $[e,x_{-2}] = 0$.
  \par
  So we can actually assume that $x$ is homogeneous of height $-3$. Then, we can write $[e,x] = w + w'$, where $w \in W, w' \in W'$ both have height $-2$. In particular, $w$ is of the form $w = \sum_{i=1}^l d_i[f_i, f_{i+l}]$ for suitable $d_{i}$.  On the other hand, the root decomposition of $w'$ can only involve the root spaces of height $-2$ with root not equal to $-\alpha_i - \alpha_{i+l}$ for any $i$ such that $1 \leq i \leq l$. By the Serre relations, it follows that $w' = \sum_{i,j=l+1}^{n} y_{ij}$ where $y_{ij}$ lies in the root space with root $-(\alpha_i + \alpha_j)$. Now note that if $-(\alpha_k + \alpha_i + \alpha_{i+l})$ is a root for $1 \leq i, k \leq l$, then the associated root space is one-dimensional. This root space is spanned by $[f_k, [f_i, f_{i+l}]]$ because $[f_i, f_{i+l}]$ is nonzero and $f_i$ and $f_k$ commute by the Serre relations.  Hence, appealing to the $Q$-grading again, we deduce that $x$ is of the form:

  \begin{align*}
    x = \sum_{k=1}^l\sum_{i,j=l+1}^{n} y_{kij} +  \sum_{k=1}^l\sum_{i=1}^l d_{ki}[f_k,[f_i, f_{i+l}]].
  \end{align*}
  where $y_{kij}$ lies is in the root space associated to the root $-(\alpha_k + \alpha_i + \alpha_j)$ and $d_{ki}$ are suitable constants. The bracket of the first sum with $e$ is $w'$ and the bracket of the second sum with $e$ is $w$. However,  $[f_k,[f_i,f_{i+l}]] = 0$ for $1 \leq k,i \leq l$ by the Serre relations, so the second sum is zero, which means that $w = 0$. And therefore $[e, x] = w' \in W'$. This shows the claim. Now, the basis $B$ is the basis of $W$ prescribed above together with any Jordan basis of $W'$.
\end{proof}
\begin{rem}
We want to emphasize that at no point in the proof of Lemma \ref{bas} did we appeal to the simplicity of $\g(A)$ (in fact $\g(A)$ is not simple when $A$ is the Cartan matrix of $\esi$). This will be important for Conjecture \ref{conj}.
\end{rem}
Let's work out an explicit example. Consider the Dynkin diagram of $\ese$, labeled as follows:

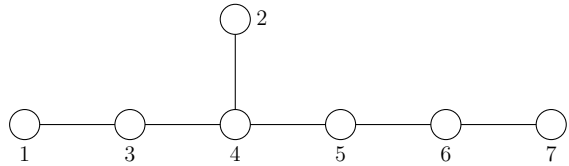
\[\dynkin[labels={1,2,3,4,5,6,7},scale=4, label distance = 5pt]  E{ooooooo}\]
Suppose we semisimplify $\ese$ with respect to $e_1 + e_2$. Then, $l = 2$, $\{i_1, i_2\} = \{1, 2\}$, and $\{j_1, j_2\} = {3, 4}$. As described in the lemma, we have the following indecomposables as direct summands, whose basis vectors commute with $e_1$ and $e_2$, and therefore $e = e_1 + e_2$:

\[\text{blocks of type} \ J_1: \ \ e_5, e_6, e_7, f_5, f_6, f_7, h_5, h_6, h_7\]
We also have these indecomposables, as their basis vectors commute with $e = e_1 + e_2$:

\[\text{other blocks of type} \ J_1: \ \ h_{3} - h_{1}, h_{4} - h_{2}.\]
These indecomposables arise because vertex 3 is connected to vertex 1 by a single arrow and vertex 4 is connected to vertex 2 by a single arrow in the Dynkin diagram:

\[\text{blocks of type} \ J_2:  \ \  e_3 \mapsto [e_1, e_3], e_4\mapsto [e_2, e_4], [f_1, f_3] \mapsto f_3, [f_2, f_4] \mapsto f_4.  \]
Finally, we have two $\sll$-triples associated to $e_1$ and $e_2$.

\[\text{blocks of type} \ J_3: \ \ \ \ f_1 \mapsto h_1 \mapsto e_1, f_2 \mapsto h_2 \mapsto e_2. \]
Notice how all the Chevalley generators are present in this decomposition and that each subspace is $Q$-graded. This concludes the example.
\par
Now, we will fix such a decomposition of $\g(A)^{(1)}$ into indecomposables as in Lemma \ref{bas}. Each Chevalley generator $e_i, f_i$ for $1 \leq i \leq n$ lies in a unique copy of $J_1$, $J_2$, or $J_3$; in the semisimplification, the image of each of these indecomposables will be an $L_1$ or $L_2$ or $0$. Therefore, we will refer to the basis vector associated to each of these copies of $L_1$ and $L_2$ as the image of the corresponding generator (in the language of \S\ref{explicit} and Proposition \ref{strucconsts}).
\par
Let $\ov{\g(A)^{(1)}}^{gen}$ denote the subquotient of $\ov{\g(A)^{(1)}}$ generated by the images of the generators, modulo the additional relation that $[x, [x,x]] = 0$ for all odd $x$ (recall that this is not automatic in characteristic $3$). We note that when $[x, [x,x]]$ is nonzero, it is purely odd. This will be important later when considering Cartan subalgebras.
\par
Recall the matrix $\widetilde{A} = (\widetilde{a}_{ij})$, which is defined to be the $(n-l) \times (n-l)$ matrix obtained from $A$ by setting $a_{j_r,j_r} = 0$ for all $r$ and deleting the row and column attached to $i_r$ for all $1 \leq r \leq l$. Now, recall the Lie algebra $\widetilde{\g}(\widetilde{A})$ defined in \S \ref{contraglie} and its generators and relations in \eqref{rels}. In particular, its upper triangular and lower triangular subalgebras are freely generated, and it is graded, see \eqref{qgrading}. We will use the letter $\widetilde{Q}$ to denote its grading and to distinguish it from the $Q$-grading on $\widetilde{\g}(A)$ and its subquotients, which incluldes $\g(A)^{(1)}$. We claim the following:

\begin{lem}\label{surj}
  There exists a surjective homomorphism from $\widetilde{\g}(\widetilde{A})^{(1)}$ to $\ov{\g(A)^{(1)}}^{gen}$.
\end{lem}
\begin{proof}
   Let us label the images of generators in the semisimplification. For $1 + l \leq i \leq n - l$, let $\widetilde{e}_{i}$ denote the basis vector of $\ov{\g(A)^{(1)}}^{gen}$ associated to the copy of $L_1$ for $\ov{e_{i+l}}$ (resp. for the $f$'s and $h$'s); for $1 \leq i \leq l$, let $\widetilde{e}_{i}$ denote the basis vector of $\ov{\g(A)^{(1)}}^{gen}$ associated to the copy of $L_2$ for $\ov{e_{i+l} \mapsto [e, e_{i+l}]}$ (resp. for the $f$'s), and let $\widetilde{h}_i$ denote the basis vector of $\ov{\g(A)^{(1)}}^{gen}$ associated to the copy of $L_1$ for $\ov{h_{i+l} - h_{i}}$.  Then, we have generators $\{\widetilde{e}_i, \widetilde{f}_i, \widetilde{h}_i\}_{1 \leq i \leq n-l}$ in $\ov{\g(A)^{(1)}}^{gen}$. The first $l$ indices are odd, and the last $n-2l$ are even.
  \par
  For $\widetilde{\g}(\widetilde{A})^{(1)}$, let us use the capital letters $E,F,H$ instead of $e,f,h$ to avoid conflict of notation with the generators of $\g(A)$. Recall that $\{E_i, F_i, H_i\}$ generate $\widetilde{\g}(\widetilde{A})^{(1)}$. Again, by definition of $\widetilde{A}$, it is an $(n-l) \times (n-l)$ matrix, where the first $l$ indices are odd and the last $n-2l$ are even. We claim that the surjection is given by the map $E_i \mapsto \widetilde{e}_i$, $F_i \mapsto \widetilde{f}_i$ and $H_i \mapsto \widetilde{h}_i$.
  \par
  To prove this, we need to check the relations in $\eqref{rels}$. We will check these using the language of \S\ref{explicit} and Proposition \ref{strucconsts}. Since these involve the bracket of two generators, let us split this into four cases based on the parity of each generator:

  \begin{enumerate}
    \item Let $1 \leq i, j \leq l$. The indices $i,j$ are both odd. Then, the bracket $[\widetilde{e}_i, \widetilde{f}_j]$ is given by 

    \[\ov{-[e_{i+l}, f_{j+l}] + [[e,e_{i+l}], [f,f_{j+l}]]} = \ov{-\delta_{ij}h_{i+l} + [[e_i, e_{i+l}], [f_j, f_{j+l}]]}.\]
    To compute the bracket in the second term of the RHS, we have by repeated applications of the Jacobi identity and the relations:

    \begin{align*}
      [[e_i, e_{i+l}], [f_j, f_{j+l}]] &= [[e_i, [f_j, f_{j+l}]], e_{i+l}] + [e_i, [e_{i+l}, [f_j, f_{j+l}]]] \\
      &= [[[e_i, f_j], f_{j+l}] + [f_j, [e_i, f_{j+l}]], e_{i+l}] \\ &\hphantom{=} + [e_i, [[e_{i+l}, f_{j}], f_{j+l}] + [f_j, [e_{i+l},f_{j+l}]]] \\
      &= [[\delta_{ij}h_i, f_{j+l}], e_{i+l}] + [e_i, [f_j, \delta_{ij}h_{i+l}]] \\ &= \delta_{ij}a_{i,i+l}h_{i+l} + \delta_{ij}a_{i+l,i}h_i \\
      &= -\delta_{ij}(h_{i+l} + h_i),
    \end{align*}
    since $a_{i+l,i} = a_{i,i+l} = -1$. Therefore, 
    
    \[-\delta_{ij}h_{i+l} + (-\delta_{ij}(h_{i+l} + h_i)) = \delta_{ij}(h_{i+l} - h_i)\] 
    in characteristic $3$, and we deduce 
    
    \[[\widetilde{e}_i, \widetilde{f}_j] = \ov{\delta_{ij}(h_{i+l} - h_i)} = \delta_{ij}\widetilde{h}_i.\]
    Now, let's check that $[\widetilde{h}_i, \widetilde{e}_j] = \widetilde{a}_{ij}\widetilde{e}_j$. Since $\widetilde{h}_i = \ov{h_{i+l} - h_i}$ and $\widetilde{e}_j = \ov{e_{j+l} \mapsto [e,e_{j+l}]}$, we need to compute the action of $h_{i+l}-h_i$ on both $e_{j+l}$ and $[e,e_{j+l}]$:

    \begin{align*}
      [h_{i+l} - h_i, e_{j+l}] &= (a_{i+l,j+l} - a_{i,j+l})e_{j+l}; \\
      [h_{i+l} - h_i, [e,e_{j+l}]] &= [h_{i+l} - h_i, [e_j,e_{j+l}]] \\
      &= [[h_{i+l} - h_i, e_j], e_{j+l}] + [e_j, [h_{i+l} - h_i, e_{j+l}]] \\ &= (a_{i+l,j} - a_{ij} + a_{i+l,j+l} - a_{i,j+l})[e_j, e_{j+l}]\\ &= (a_{i+l,j} - a_{ij} + a_{i+l,j+l} - a_{i,j+l})[e, e_{j+l}].
    \end{align*}
    If $i \neq j$, $a_{i+l,j+l} = \widetilde{a}_{ij}$ and $a_{ij} = a_{i+l,j} = a_{i,j+l} = 0$, and the coefficient of the RHS simplifies to $\widetilde{a}_{ij}$ in both equations. If $i = j$, then $a_{i+l,i+l} = a_{ii} = 2$ and $a_{i+l,i} = a_{i,i+l} = -1$, and again the coefficient of the RHS simplifies to $0 = \widetilde{a}_{ii}$ in both equations (remember we are in characteristic $3$). This shows that $[\widetilde{h}_i, \widetilde{e}_j] = \widetilde{a}_{ij}\widetilde{e}_{j}$. A similar argument goes through for the other relation $[\widetilde{h}_i, \widetilde{f}_j] = -\widetilde{a}_{ij}\widetilde{f}_j$.
    \par
    Finally, the Cartan subalgebra will continue to be commutative, so the last relation holds as well. Therefore, we deduce that the relations in \eqref{rels} hold between generators that correspond to odd indices.

    \item Let $1 \leq i \leq l$ and $1+l \leq j \leq n-l$. The index $i$ is odd and the index $j$ is even. We proceed similarly to the first case. First, we notice that $[\widetilde{e}_i, \widetilde{f}_j] = 0$ as this must lie in the Cartan subalgebra, which is purely even, but the first vector is odd and the second vector is even. And indeed, the condition on $i$ and $j$ ensures $i\neq j$, so we do have $[\widetilde{e}_i, \widetilde{f}_j] = \delta_{ij}\widetilde{h}_i$ in this case.
    \par
    Next, let's check $[\widetilde{h}_i, \widetilde{e}_j] = \widetilde{a}_{ij}\widetilde{e}_j$. Since $\widetilde{h}_i = \ov{h_{i+l} - h_i}$ and $\widetilde{e}_j = \ov{e_{j+l}}$, we compute 
    
    \[[h_{i+l} - h_i, e_{j+l}] = (a_{i+l,j+l} - a_{i,j+l})e_{j+l} = a_{i+l,j+l}e_{j+l} = \widetilde{a}_{ij}e_{j+l},\] 
    from which the desired result follows. A similar argument shows that $[\widetilde{h}_i, \widetilde{f}_j] = -\widetilde{a}_{ij}{f}_j$. Finally, the last relation holds again as the commutativity of the Cartan subalgebra is preserved. This shows relations \eqref{rels} in this case.

    \item Let $1+l \leq i \leq n-l$ and $1 \leq j \leq l$. The index $i$ is even and the index $j$ is odd. A similar argument to the previous case shows $[\widetilde{e}_i, \widetilde{f}_j] = \delta_{ij}\widetilde{h}_i = 0$. Now, let's compute $[\widetilde{h}_i, \widetilde{e}_j]$. Since $\widetilde{h}_i = \ov{h_{i+l}}$ and $\widetilde{e}_j = \ov{e_{j+l} \mapsto [e, e_{j+l}]}$, we check the action of $h_{i+l}$. Therefore, we have

    \begin{align*}
      [h_{i+l}, e_{j+l}] &= a_{i+l,j+l}e_{j+l} = \widetilde{a}_{ij}e_{j+1}; \\
      [h_{i+l}, [e,e_{j+l}]] &= [h_{i+l}, [e_j,e_{j+l}]] \\
      &=  [h_{i+l}, e_j], e_{j+l}] + [e_j, [h_{i+l}, e_{j+l}]]\\
      &= (a_{i+l,j} + a_{i+l,j+l})[e_j,e_{j+l}] = a_{i+l,j+l}[e, e_{j+l}] \\
      &= \widetilde{a}_{ij}[e,e_{j+l}],
    \end{align*}
    from which we deduce $[\widetilde{h}_i, \widetilde{e}_j] = \widetilde{a}_{ij}\widetilde{e}_j$. A similar argument shows that $[\widetilde{h}_i, \widetilde{f}_j] = -\widetilde{a}_{ij}\widetilde{f}_j$.  Finally, the last relation holds again as the commutativity of the Cartan subalgebra is preserved. This shows relations \eqref{rels} in this case.
    \item $1+l \leq i,j \leq n-l$. The indices $i,j$ are both even. This is the easiest case, and the relations in \eqref{rels} follow immediately.
  \end{enumerate}
  We deduce that the generators $\{\widetilde{e}_i, \widetilde{f}_i, \widetilde{h}_i\}$ satisfy the same relations (and actually, they satisfy more relations) as the generators $\{E_i, F_i, H_i\}$ of $\widetilde{\g}(\widetilde{A})^{(1)}$. This gives the desired surjection.
\end{proof}

We can now state the main theorem. 
\begin{thm}\label{maint}
If $\sdim \ov{\g(A)^{(1)}} = \sdim \g(\widetilde{A})^{(1)}$, then $\ov{\g(A)^{(1)}}^{gen} = \ov{\g(A)^{(1)}}$ and $\ov{\g(A)^{(1)}}$ is isomorphic to $\g(\widetilde{A})^{(1)}$.
\end{thm}

\begin{proof}
  By Lemma \ref{surj}, we know there is a surjective, $\widetilde{Q}$-grading-preserving homomorphism from $\widetilde{\g}(\widetilde{A})^{(1)}$ to $\ov{\g(A)^{(1)}}^{gen}$, and both have a Cartan subalgebra of the same dimension. Therefore, the kernel of this homomorphism is graded and trivially intersects the Cartan subalgebra, and so this gives a well-defined surjection from $\ov{\g(A)^{(1)}}^{gen}$ to $\g(\widetilde{A})^{(1)}$. On the other hand, by the dimension hypothesis, if the dimension of $\ov{\g(A)^{(1)}}^{gen}$ is less than the dimension of $\ov{\g(A)^{(1)}}$, the surjection cannot exist. Therefore, we deduce $\ov{\g(A)^{(1)}}^{gen} = \ov{\g(A)^{(1)}}$ and that this surjection must be an isomorphism.
\end{proof}

\begin{rem}
It would be interesting to see how this theorem can be generalized (see Conjecture \ref{conj}), both for higher characteristic and more general Cartan matrices.
\end{rem}

\subsection{Other Examples of Semisimplification }
In this section, we offer some other examples of semisimplification, both for completeness and to highlight some potential pitfalls. For instance, in the proof of Theorem \ref{maint}, we carefully showed that certain relations were satisfied, although at first glance these were the ``obvious'' relations to be satisfied; in general, this cannot be expected.

\subsubsection{Example 1.} Consider the additive group scheme $\mathbb{G}_a$ over $\mathbb{C}$. The representation category $\Rep \mathbb{G}_a$ is a symmetric tensor category, and its objects are finite-dimensional nilpotent $\mathbb{C}[t]$-modules. Any such module is the direct sum of Jordan blocks, and the tensor product of Jordan blocks is described by the usual Clebsch-Gordan rule. It follows that the semisimplification of this category is $\Rep SL_2(\mathbb{C})$, where the Jordan block of dimension $n$ semisimplifies to the $n$-dimensional $SL_2(\mathbb{C})$-module (we are effectively deleting the maps between Jordan blocks of different sizes and nilpotent endomorphisms, cf. \cite{etingof2019semisimplification}).
\par
Any Lie algebra $\mathfrak{g}$ in $\Rep \mathbb{G}_a$ is therefore a Lie algebra equipped with a nilpotent derivation $d$, and semisimplifying gives a Lie algebra $\ov{\g}$ with an action of $\sll(\C)$. If $\g$ is semisimple, then $d$ is inner and we can write $d(x) = [e,x]$ for some $e \in \g$. By the Jacobson-Morozov lemma, $e$ can be included in an $\sll$-triple $\{e,f,h\}$, and the semisimplification of $\g$ is isomorphic to $\g$ with the action of $\sll$ prescribed by this triple.
\par
On the other hand, if $\g$ is not semisimple, the action of $d$ may not extend to an action of $\sll$ by derivations (even if $d$ is inner). This can result in $\ov{\g}$ not being isomorphic to $\g$ unlike the previous case, and in fact, $\ov{\g}$ may even be abelian. In general, $\ov{\g}$ is the associated graded algebra of $\g$ under the Deligne filtration by $d$, which on a vector space $V$ is defined by 

\[F_k V= \bigoplus_{j-i=k} \ker d^i\cap \im d^j.\] 
If $\ov{\g}$ is semisimple, this filtration extends to a grading by eigenvalues of $h$, so $\mathrm{gr}(\g) = \g$. But when $\ov{\g}$ is not semisimple, it may not extend to such a grading, and hence we may not have such an isomorphism.
\par
Here is an explicit example. Consider the three-dimensional Heisenberg Lie algebra spanned by $x,y,z$ with $z = [x,y]$ the central element. Let $d$ be the derivation given by $\ad x$. This does not extend to an $\sll$-action, and the semisimplification is abelian, which we'll explicitly check: using the notation of Jordan blocks from \S\ref{Verp} in the obvious way for characteristic $0$, we have a copy of $J_1$ spanned by $x$ and a copy of $J_2$ given by $y \mapsto z$. The bracket $[x,y] = z$ and the bracket $[x,z] = 0$ are encapsulated by a morphism $J_1 \otimes J_2 \rightarrow J_2$ which is negligible, so it becomes zero in the semisimplification.

\subsubsection{Example 2 (Duflo-Serganova, cf. \cite{Duflo2008ONAV}).} Consider the representation category $\mathcal{C} = \Rep \mathbb{G}_a^{0|1}$. A Lie algebra in $\mathcal{C}$ is a Lie superalgebra with an odd derivation $d$ such that $d^2 = 0$. The semisimplification of $\mathcal{C}$ is $\sVec_\C$ and the semisimplification of $\g$ is the cohomology of $d$, which is a new Lie superalgebra.
\subsubsection{Example 3 (Entova-Aizenbud and Serganova, cf. \cite{entovaaizenbud2020jacobsonmorozov}).} Consider the representation category $\mathcal{C}$ of the affine supergroup scheme $\mathbb{G}_a^{1|1}$, whose coordinate ring is the symmetric algebra $S((\mathbb{C}^{1|1})^*)$. A Lie algebra $\g$ in $\mathcal{C}$ is a Lie superalgebra with an odd, nilpotent derivation $d$. The semisimplification of $\mathcal{C}$ is $\Rep \mathfrak{osp}_{1|2}$, and $\ov{\g}$ is a Lie superalgebra with an action of $\mathfrak{osp}_{1|2}$. If $\g$ is \textit{quasireductive} (i.e. the even part is reductive and acts semisimply on the odd part) and $d = \ad e$ is inner, where $e$ is a \textit{neat element} (i.e. $[e,e]$ acts as the sum of odd-dimensional Jordan blocks on every finite-dimensional $\g$-module), then it is shown that $\ov{\g} = \g$ and $e$ extends to an action of $\mathfrak{osp}_{1|2}$. However, this is nontrivial and only true under these conditions. This is in a sense a super analog of Example 1 above.

\section{Exceptional Simple Lie Superalgebras in Characteristic $3$}\label{char3constructions}
In this section, we use Theorem \ref{maint} above to construct exceptional simple Lie superalgebras via semisimplification, except in two cases. We are able to construct all ten simple Elduque and Cunha Lie superalgebras, the Brown Lie superalgebra $\mathfrak{brj}(2; 3)$ and the Elduque Lie superalgebra $\mathfrak{el}(5; 3)$. We will use the notation $\g(a, b)$ to denote  the Lie superalgebra occupying the $(a, b)$-th slot in the Elduque Supermagic Square (cf. \cite{bouarroudj2009classification, cunha2007extended, elduque2006new, cunha2006extended}).
\par
For completeness, we include a description of the even part and odd part of each Lie superalgebra. The bracket makes the odd part a module over the even part. Each even part is either a contragredient Lie algebra $\g(C)$ with Cartan matrix $C$ or it is the quotient of $\g(C)^{(1)}$ by its center. If the Cartan matrix $C$ of the even part is invertible, we will use the notation $L(\omega_i)$ to denote the Weyl module whose highest weight is a fundamental weight $\omega_i$, whose labels will follow that of Bourbaki (cf. \cite{bourbaki2008lie} and \cite{jantzen2003representations} for a definition of a Weyl module). The exception to this will be the tautological module $\KK^n = L(\omega_1)$ over $\mathfrak{sl}_n$. If $C$ is not invertible, then we will explicitly describe the module. 

\subsection{Connection to Prior Constructions}
Before we show how to construct the Lie superalgebras above, we describe a setup already known in the literature that is closely connected to semisimplification. This serves as motivation as to why semisimplification might produce many of the exceptional Lie superalgebras. In \cite{elduque2006new} and \cite{elduque2009models}, a procedure is described by which one can start with a so called \textit{symplectic triple system} $T$ over a field $\KK$ to produce a Lie algebra $\g$ containing a subalgebra $\mathfrak{s}$ isomorphic to $\sll$, such that $\g = D \oplus \KK^2 \otimes T \oplus \mathfrak{s}$. Here $D$ is the centralizer of $\mathfrak{s}$ which acts on $T$, and $T$ is the multipicity space of the tautological module $\KK^2$ over $\sll$. The triple product in $T$ induces a $D$-invariant symmetric bilinear map $T \times T \rightarrow D$ giving the bracket in $\g$. Moreover, $\g$ is a Lie algebra in $\Rep_\KK \sll$. Suppose now that the characteristic of $\KK$ is $p = 3$. Assembling this data together gives a Lie superalgebra $\widetilde{\g} = D \oplus T$. 
\par
On the other hand, if one forgets the action of $f,h$ after fixing an isomorphism $\sll \rightarrow \mathfrak{s}$, we can realize $\g$ as an object in $\Rep \bal_3$ with respect to the action of $e$ (here $\{e,f,h\}$ is the usual basis of $\sll$). Then, if we semisimplify, we get precisely the Lie superalgebra described above (essentially by definition, as relations in $\Ver_p$ boil down to linear algebra), up to a natural isomorphism.
\par
This method is how some of the Elduque and Cunha Lie superalgebras were first constructed, namely $\g(1, 6), \g(2, 6), \g(4, 6)$, and $\g(8, 6)$. Also, $\mathfrak{brj}_{2;3}$ can be constructed this way. For these Lie superalgebras, this reflects what is going on behind the scenes when phrased using the language of symmetric tensor categories.
\par
All of the Elduque and Cunha Lie superalgebras can be constructed using the Elduque Supermagic square. The authors of \cite{elduque2022} give conceptual reasoning as to why this method and semisimplification are related . In particular, Lie superalgebras in the Elduque Supermagic Square can be obtained by semisimplifying exceptional Lie algebras (realized as Lie algebras in $\Rep \Z/3Z$) in the fourth row of Freudenthal's Magic Square (c.f. section $4$ in \cite{elduque2022}).

\subsection{Constructing $\mathfrak{brj}_{2;3}$ from $\mathfrak{br}_3$}
In this section, we will construct the Brown Lie superalgebra $\mathfrak{brj}_{2; 3}$. The ``$3$" in the index of $\mathfrak{brj}_{2;3}$ is used to distinguish it from its characteristic $5$ analog, which we do not discuss in this paper. This is a simple contragredient Lie superalgebra of superdimension $(10|8)$ with a Cartan matrix of full rank and parity set:

\begin{align*}
  \widetilde{A} = \begin{pmatrix}
    0 & -1 \\
    -1 & \ov{0}
  \end{pmatrix}; \ \ \ I = \{1,0\}.
\end{align*}
The even part of $\mathfrak{brj}_{2;3}$ is the $10$-dimensional rank $2$ Brown Lie algebra $\mathfrak{br}_2$, which is simple and unique to characteristic $3$. It has Cartan matrix

\begin{align*}
  \begin{pmatrix}
     2 & -1 \\
     -1 & \ov{0}
  \end{pmatrix}; \ \ \ I = \{0,0\}.
\end{align*}
The odd part is the $8$-dimensional simple module $L(2\omega_1)$ over $\mathfrak{br}_2$. To apply our main theorem, we consider the Lie algebra with Cartan matrix

\begin{align*}
  A = \begin{pmatrix}
    2 & -1 & 0 \\
    -1 & 2 & -1 \\
    0 & -1 & \ov{0}
  \end{pmatrix}; \ \ \ I = \{0,0,0\}.
\end{align*}
This is the full-rank Cartan matrix of the rank $3$ Brown Lie algebra $\mathfrak{br}_3$, which is $29$-dimensional, simple, and unique to characteristic $3$. The Lie algebra $\mathfrak{br}_3$ has the following Dynkin diagram, labeled in accordance with the Cartan matrix:

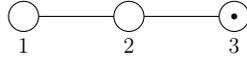
\[    \begin{dynkinDiagram}[labels={1,2,3},scale=4, mark=o, label distance = 5pt]A3  
  \fill[black] (root 3) circle (0.01cm);
  \end{dynkinDiagram}\]
Here, we use a special node to indicate the node corresponding to the last index, as $b_{33} = 0$ but the index $3$ is even. The Lie algebra $\mathfrak{br}_3$ can be realized as an object in $\Rep \bal_3$ with respect to the adjoint action of $e_1$ and decomposes as $10J_1 \oplus 8 J_2 \oplus J_3$, which can be checked using the software $SuperLie$. Therefore, comparing dimensions, by Theorem \ref{maint}, we have:

\begin{cor}
  The semisimplification of $\mathfrak{br}_3$ as an object in $\Rep \bal_3$ under the adjoint action of $e_1$ is $\mathfrak{brj}_{2;3}$.
\end{cor}

\subsection{Constructing $\g(1, 6)$ from $\ff$}
In this section, we will construct the Elduque and Cunha Lie superalgebra $\g(1, 6)$. This is a simple contragredient Lie superalgebra of superdimension $(21|14)$ with a Cartan matrix of full rank and parity set:

\begin{align*}
  \widetilde{A} = \begin{pmatrix}
    2 & -1 & 0 \\
    -1 & 2 & -2 \\
    0 & -1 & 0
  \end{pmatrix}; \ \ \ I = \{0,0,1\}.
\end{align*}
The even part of $\g(1, 6)$ is $\mathfrak{sp}_6$, and the odd part is its $14$-dimensional simple module $L(\omega_3)$. To apply our main theorem, we consider the Lie algebra with Cartan matrix

\begin{align*}
  A = \begin{pmatrix}
    2 & -1 & 0 & 0 \\
    -1 & 2 & -2 & 0 \\
    0 & -1 & 2 & -1 \\
    0 & 0 & -1 & 2
  \end{pmatrix}.
\end{align*}
This is the full-rank Cartan matrix of the $52$-dimensional simple Lie algebra $\ff$. The Lie algebra $\ff$ has the following Dynkin diagram, labeled in accordance with the Cartan matrix:

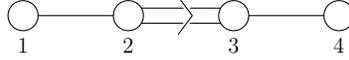
\[\dynkin[labels={1,2,3,4},scale=4, arrow shape/.style={-{jibladze[length=7pt]}}, label distance = 5pt]  F{oooo}\]
The Lie algebra $\ff$ can be realized as an object in $\Rep \bal_3$ with respect to the adjoint action of $e_4$ and decomposes as $21J_1 \oplus 14 J_2 \oplus J_3$, which can be checked using the software $SuperLie$. Therefore, comparing dimensions, by Theorem \ref{maint}, we have:

\begin{cor}
  The semisimplification of $\ff$ as an object in $\Rep \bal_3$ under the adjoint action of $e_4$ is $\g(1, 6)$.
\end{cor}

\subsection{Lie Superalgebras Arising from $\esi$}
In characteristic $3$, the Lie algebra $\esi = \g(A)$ is $79$-dimensional and has Cartan matrix:

\begin{equation}\label{esicart}
  A = \begin{pmatrix}
    2 & 0 & -1 & 0 & 0 & 0 \\
    0 & 2 & 0 & -1 & 0 & 0\\
    -1 & 0 & 2 & -1 & 0 & 0\\
    0 & -1 & -1 & 2 & -1 & 0 \\
    0 & 0 & 0 & -1 & 2 & -1 \\
    0 & 0 & 0 & 0 & -1 & 2
  \end{pmatrix}.
\end{equation}
In particular, in characteristic $3$, $A$ has rank $5$. Therefore, $\g(A) \neq \g(A)^{(1)}$ as discussed in \S\ref{contraglie}. The Lie subalgebra $\g(A)^{(1)}$ is a $78$-dimensional Lie algebra with one-dimensional center $\mathfrak{z}$ in the Cartan subalgebra; quotienting out by the center gives a $77$-dimensional simple Lie algebra.
\par
The Lie algebra $\esi$ has the following Dynkin diagram, labeled in accordance with the Cartan matrix:
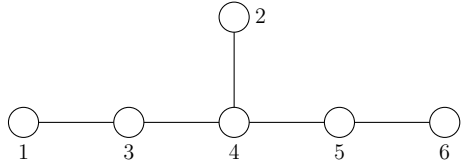
\[\dynkin[labels={1,2,3,4,5,6},scale=4, label distance = 5pt]  E{oooooo}\]

\subsubsection{Constructing $\g(2, 3)$}
In this section, we will construct the derived algebra of the Elduque and Cunha Lie superalgebra $\g(2, 3) = \g(\widetilde{A})$. The Lie superalgebra $\g(2, 3)$ has Cartan matrix and parity set:

\begin{align*}
  \widetilde{A} = \begin{pmatrix}
    0 & -1 & 0 \\
    -1 & 0 & -1 \\
    0 & -1 & 0
  \end{pmatrix}; \ \ \ I = \{1,1,1\}.
\end{align*}
Like $A$ above in \eqref{esicart}, $\widetilde{A}$ does not have full rank. Using SuperLie, we can check that $\g(\widetilde{A})^{(1)}$ has superdimension $(11|14)$ and a one-dimensional center $\mathfrak{c}$ lying in the Cartan subalgebra. Quotienting out by this center gives a simple Lie superalgebra. The even part of $\g(2, 3)^{(1)}/\mathfrak{c}$ is $\mathfrak{psl}_3\oplus \sll$, and its odd part is the $\mathfrak{psl}_3\otimes \KK^2$; here $\mathfrak{psl}_3$ acts on $\mathfrak{psl}_3$ by the adjoint representation and $\KK^2$ is the tautological $\sll$-module.
\par
By comparing Cartan matrices, we realize $\g(A)^{(1)}$ as an object in $\Rep \bal_3$ with respect to $e_1 + e_2 + e_6$, where it decomposes as $\g(A)^{(1)} = 11J_1 \oplus 14J_2 \oplus 13J_3 $. Then, the hypothesis of Theorem \ref{maint} is satisfied, and we deduce that the semisimplification of $\g(A)^{(1)}$ with respect to $e_1 + e_2 + e_6$ is $\g(\widetilde{A})^{(1)}$. We can then mod out by the centers to deduce the following:

\begin{cor}
  The semisimplification of the simple Lie algebra $\esi^{(1)}/\mathfrak{z}$ as an object in $\Rep \bal_3$ under the adjoint action of $e_1 + e_2 + e_6$ is $\g(2, 3)^{(1)}/\mathfrak{c}$.
\end{cor}
\subsubsection{Constructing $\g(3, 3)$}
In this section, we will construct the derived algebra of the Elduque and Cunha Lie superalgebra $\g(3, 3)$.  The Lie superalgebra $\g(3, 3)$ has Cartan matrix and parity set:

\begin{align*}
  \widetilde{A} = \begin{pmatrix}
    0 & -1 & 0  & 0\\
    -1 & 0 & -1 & 0 \\
    0 & -1 & 2 & -1 \\
    0 & 0 & -1 & 2
  \end{pmatrix}; \ \ \ I = \{0,0,1,1\}.
\end{align*}
Like $A$ above in \eqref{esicart}, $\widetilde{A}$ does not have full rank. Using SuperLie, we can check that $\g(\widetilde{A})^{(1)}$ has superdimension $(22|16)$ and a one-dimensional center $\mathfrak{c}$ lying in the Cartan subalgebra. Quotienting out by this center gives a simple Lie superalgebra. The even subalgebra of $\g(3, 3)^{(1)}/\mathfrak{c}$ is $\mathfrak{o}_7$, and the odd part, as a module over $\mathfrak{o}_7$, is $L(\omega_3) \oplus L(\omega_3)$, where in particular $L(\omega_3)$ is the $8$-dimensional spinor module over $\mathfrak{o}_7$.
\par
By comparing Cartan matrices, we realize $\g(A)^{(1)}$ as an object in $\Rep \bal_3$ with respect to $e_1 + e_2$, where it decomposes as $\g(A)^{(1)} = 22J_1 \oplus 16J_2 \oplus 8J_3$. Then, the hypothesis of Theorem \ref{maint} is satisfied, and we deduce that the semisimplification of $\g(A)^{(1)}$ with respect to $e_1 + e_2$ is $\g(\widetilde{A})^{(1)}$. We can then mod out by the centers to deduce the following:

\begin{cor}
  The semisimplification of $\esi^{(1)}/\mathfrak{z}$ as an object in $\Rep \bal_3$ under the adjoint action of $e_1 + e_2$ is $\g(3, 3)^{(1)}/\mathfrak{c}$.
\end{cor}

\subsubsection{Constructing $\g(2, 6)$}
In this section, we will construct the derived algebra of the Elduque and Cunha Lie superalgebra $\g(2, 6)$. The Lie superalgebra $\g(2, 6)$ has Cartan matrix and parity set:

\begin{align*}
  \widetilde{A} = \begin{pmatrix}
    2 & -1 & 0 & 0 & 0\\
    -1 & 2 & -1 & 0 & 0 \\
    0 & -1 & 0 & -1 & 0 \\
    0 & 0 & -1 & 2 & -1 \\
    0 & 0 & 0 & -1 & 2
  \end{pmatrix}; \ \ \ I = \{0,0,1,0,0\}.
\end{align*}
Like $A$ above in \eqref{esicart}, $\widetilde{A}$ does not have full rank. Using SuperLie, we can check that $\g(\widetilde{A})^{(1)}$ has superdimension $(35|20)$ and a one-dimensional center $\mathfrak{c}$ lying in the Cartan subalgebra. Quotienting out by this center gives a simple Lie superalgebra. The even subalgebra of $\g(2, 6)^{(1)}/\mathfrak{c}$ is $\mathfrak{psl}_6$ and the odd part is the $20$-dimensional simple module $\bigwedge^3(\KK^6)$, which is the third exterior power of the tautological module $\KK^6$ over $\mathfrak{sl}_6$. The action of $\mathfrak{psl}_6$ on $\bigwedge^3(\KK^6)$ is given as follows. If $Z$ is the one-dimensional center of $\mathfrak{sl}_6$ in characteristic $3$, then for each $x + Z \in \mathfrak{psl}_6$, choose a suitable lift $\widehat{x}$ in $\mathfrak{sl}_6$, and define $(x+Z)\cdot v = \widehat{x}v$ for all $v \in \KK^6$, where the RHS is usual matrix multiplication. Then, on $\bigwedge^3(\KK^6)$, we have 

\[(x+Z)\cdot(v_1 \wedge v_2 \wedge v_3) \coloneqq \widehat{x}v_1 \wedge v_2 \wedge v_3 + v_1 \wedge \widehat{x}v_2 \wedge v_3 +  v_1 \wedge v_2 \wedge \widehat{x}v_3\]
for any $v_1,v_2,v_3 \in \KK^6$. This defines a well-defined Lie algebra action of $\mathfrak{psl}_6$ on $\bigwedge^3(\KK^6)$, because if $\widehat{y}$ is another lift of $x+Z$, then $\widehat{x} - \widehat{y}$ is central, so it acts as a scalar on $\KK^6$. It follows that the action on $\bigwedge^3(\mathbb{K}^6)$ will then differ by three times this scalar, which in characteristic $3$ is zero.
\par
By comparing Cartan matrices, we realize $\g(A)^{(1)}$ as an object in $\Rep \bal_3$ with respect to $e_2$, where it decomposes as $\g(A)^{(1)} = 35J_1 \oplus 20J_2 \oplus J_3$. Then, the hypothesis of Theorem \ref{maint} is satisfied, and we deduce that the semisimplification of $\g(A)^{(1)}$ with respect to $e_2$ is $\g(\widetilde{A})^{(1)}$. We can then mod out by the centers to deduce the following:

\begin{cor}
  The semisimplification of $\esi^{(1)}/\mathfrak{z}$ as an object in $\Rep \bal_3$ under the adjoint action of $e_2$ is $\g(2, 6)^{(1)}/\mathfrak{c}$.
\end{cor}

\subsection{Lie Superalgebras Arising from $\ese$}
Recall that $\ese$ is the $133$-dimensional simple Lie algebra with Cartan matrix:

\begin{align*}
  A = \begin{pmatrix}
    2 & 0 & -1 & 0 & 0 & 0 & 0\\
    0 & 2 & 0 & -1 & 0 & 0 & 0\\
    -1 & 0 & 2 & -1 & 0 & 0 & 0\\
    0 & -1 & -1 & 2 & -1 & 0 & 0 \\
    0 & 0 & 0 & -1 & 2 & -1 & 0\\
    0 & 0 & 0 & 0 & -1 & 2 & -1 \\
    0 & 0 & 0 & 0 & 0 & -1 & 2
  \end{pmatrix}.
\end{align*}
It has the following Dynkin diagram, labeled in accordance with the Cartan matrix:
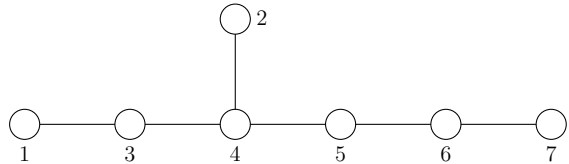
\[\dynkin[labels={1,2,3,4,5,6,7},scale=4, label distance = 5pt]  E{ooooooo}\]

\subsubsection{Constructing $\g(4, 3)$}
In this section, we will construct the Elduque and Cunha Lie superalgebra $\g(4, 3)$. This is a simple contragredient Lie superalgebra with a Cartan matrix and parity set:

\begin{align*}
  \widetilde{A} = \begin{pmatrix}
    0 & -1 & 0 & 0 \\
    -1 & 0 & -1 & 0  \\
    0 & -1 & 2 & -1 \\
    0 & 0 & -1 & 0
  \end{pmatrix}; \ \ \ I = \{1,1,0,1\}.
\end{align*}
This is a Lie superalgebra of superdimension $(24|26)$. It has an even subalgebra $\mathfrak{sp}_6 \oplus \sll$ and the module $L(\omega_2)'\otimes \KK^2$ over the even subalgebra is its odd part. Here, $L(\omega_2)$ is the $14$-dimensional Weyl module over $\mathfrak{sp}_6$ of highest weight $\omega_2$, which contains a copy of the trivial module as a submodule. We let $L(\omega_2)'$ denote the $13$-dimensional simple module which is the quotient of $L(\omega_2)$ by this submodule, and $\KK^2$ is the tautological $\sll$-module. 
\par
By comparing Cartan matrices, we realize $\ese$ as an object in $\Rep \bal_3$ with respect to $e_1 + e_2 + e_7$, where it decomposes as $\ese = 24J_1 \oplus 26J_2 \oplus 19J_3 $. Then, the hypothesis of Theorem \ref{maint} is satisfied, and we have:

\begin{cor}
  The semisimplification of $\ese$ as an object in $\Rep \bal_3$ under the adjoint action of $e_1  + e_2 + e_7$ is $\g(4, 3)$.
\end{cor}

\subsubsection{Constructing $\mathfrak{el}(5; 3)$}
In this section, we will construct the Elduque Lie superalgebra $\mathfrak{el}(5; 3)$. This is a simple contragredient Lie superalgebra with a Cartan matrix and parity set:

\begin{align*}
  \widetilde{A} = \begin{pmatrix}
  2 & 0 & -1 & 0 & 0 \\
  0 & 0 & -1 & 0 & 0  \\
  -1 & -1 & 2 & -1 & 0 \\
  0 & 0 & -1 & 2 & -1 \\
  0 & 0 & 0 & -1 & 0 \\
  \end{pmatrix}; \ \ \ I = \{0,1,0,0,1\}.
\end{align*}
This is a Lie superalgebra of superdimension $(39|32)$. Its even subalgebra is $\mathfrak{o}_{9} \oplus \sll$, and its odd part is $L(\omega_4) \otimes \KK^2$, where $L(\omega_4)$ is the $16$-dimensional simple Weyl module over $\mathfrak{o}_9$ of highest weight $\omega_4$ and $\KK^2$ is the tautological $\sll$-module. By comparing Cartan matrices, we realize $\ese$ as an object in $\Rep \bal_3$ with respect to $e_1 + e_7$, where it decomposes as $\ese = 39J_1 \oplus 32J_2 \oplus 10J_3 $. Then, the hypothesis of Theorem \ref{maint} is satisfied, and we have:

\begin{cor}
  The semisimplification of $\ese$ as an object in $\Rep \bal_3$ under the adjoint action of $e_1 +e_7$ is $\mathfrak{el}(5; 3)$.
\end{cor}

\subsubsection{Constructing $\g(4, 6)$}
In this section, we will construct the Elduque and Cunha Lie superalgebra $\g(4, 6)$. This is a simple contragredient Lie superalgebra with a Cartan matrix and parity set:

\begin{align*}
  \widetilde{A} = \begin{pmatrix}
  2 & 0 & -1 & 0 & 0 & 0 \\
  0 & 0 & -1 & 0 & 0 & 0 \\
  -1 & -1 & 2 & -1 & 0 & 0 \\
  0 & 0 & -1 & 2 & -1 & 0 \\
  0 & 0 & 0 & -1 & 2 & -1 \\
  0 & 0 & 0 & 0 & -1 & 2
  \end{pmatrix}; \ \ \ I = \{0,1,0,0,0,0\}.
\end{align*}
This is a Lie superalgebra of superdimension $(66|32)$. Its even subalgebra is $\mathfrak{o}_{12}$, and its odd part is the $32$-dimensional simple Weyl module $L(\omega_5)$. By comparing Cartan matrices, we realize $\ese$ as an object in $\Rep \bal_3$ with respect to $e_1$, where it decomposes as $\ese = 66J_1 \oplus 32J_2 \oplus J_3 $. Then, the hypothesis of Theorem \ref{maint} is satisfied, and we have:

\begin{cor}
  The semisimplification of $\ese$ as an object in $\Rep \bal_3$ under the adjoint action of $e_1$ is $\g(4, 6)$.
\end{cor}

\subsection{Lie Superalgebras Arising From $\ee$}
Recall that $\ee$ is the $248$-dimensional simple Lie algebra with Cartan matrix:

\begin{align*}
  A = \begin{pmatrix}
    2 & 0 & -1 & 0 & 0 & 0 & 0 & 0\\
    0 & 2 & 0 & -1 & 0 & 0 & 0 & 0\\
    -1 & 0 & 2 & -1 & 0 & 0 & 0 & 0\\
    0 & -1 & -1 & 2 & -1 & 0 & 0 & 0\\
    0 & 0 & 0 & -1 & 2 & -1 & 0 & 0\\
    0 & 0 & 0 & 0 & -1 & 2 & -1 & 0 \\
    0 & 0 & 0 & 0 & 0 & -1 & 2 & -1 \\
    0 & 0 & 0 & 0 & 0 & 0 & -1 & 2
  \end{pmatrix}.
\end{align*}
It has the following Dynkin diagram, labeled in accordance with the Cartan matrix:
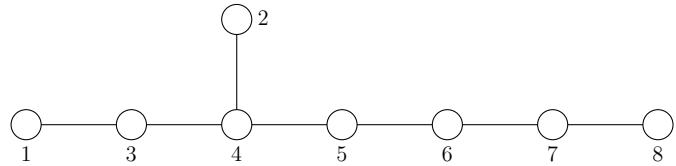
\[\dynkin[labels={1,2,3,4,5,6,7,8},scale=4, label distance = 5pt]  E{oooooooo}\]
\subsubsection{Constructing $\g(8, 3)$}
In this section, we will construct the Elduque and Cunha Lie superalgebra $\g(8, 3)$. This is a simple contragredient Lie superalgebra with a Cartan matrix and parity set:

\begin{align*}
  \widetilde{A} = \begin{pmatrix}
  0 & -1 & 0 & 0 & 0  \\
  -1 & 0 & -1 & 0 & 0 \\
  0 & -1 & 2 & -1 & 0 \\
  0 & 0 & -1 & 2 & -1 \\
  0 & 0 & 0 & -1 & 0 \\
  \end{pmatrix}; \ \ \ I = \{1,1,0,0,1\}.
\end{align*}
This is a Lie superalgebra of superdimension $(55|50)$. It has an even subalgebra $\mathfrak{f}_4 \oplus \sll$ and the module $L(\omega_4)'\otimes \KK^2$ over the even subalgebra is its odd part. Here, $L(\omega_4)$ is the $26$-dimensional Weyl module over $\ff$ of highest weight $\omega_4$, which contains a copy of the trivial module as a submodule. We let $L(\omega_4)'$ denote the $25$-dimensional simple module which is the quotient of $L(\omega_4)$ by this submodule, and $\KK^2$ is the tautological $\sll$-module.
\par
By comparing Cartan matrices, we realize $\ee$ as an object in $\Rep \bal_3$ with respect to $e_1 + e_2 + e_8$, where it decomposes as $\ee = 55J_1 \oplus 50J_2 \oplus 31J_3$. Then, the hypothesis of Theorem \ref{maint} is satisfied, and we have:

\begin{cor}
  The semisimplification of $\ee$ as an object in $\Rep \bal_3$ under the adjoint action of $e_1 + e_2 + e_8$ is $\g(8, 3)$.
\end{cor}

\subsubsection{Constructing $\g(6, 6)$}
In this section, we will construct the Elduque and Cunha Lie superalgebra $\g(6, 6)$. This is a simple contragredient Lie superalgebra with a Cartan matrix and parity set:

\begin{align*}
  \widetilde{A} = \begin{pmatrix}
  0 & -1 & 0 & 0 & 0 & 0\\
  -1 & 0 & -1 & 0 & 0 & 0\\
  0 & -1 & 2 & -1 & 0 & 0\\
  0 & 0 & -1 & 2 & -1 & 0 \\
  0 & 0 & 0 & -1 & 2 & -1 \\
  0 & 0 & 0 & 0 & -1 & 2
  \end{pmatrix}; \ \ \ I = \{1,1,0,0,0,0\}.
\end{align*}
This is a Lie superalgebra of superdimension $(78|64)$. Its even subalgebra is $\mathfrak{o}_{13}$, and its odd part is the $64$-dimensional simple spinor module $L(\omega_6)$. By comparing Cartan matrices, we realize $\ee$ as an object in $\Rep \bal_3$ with respect to $e_1 + e_2$, where it decomposes as $\ee = 78J_1 \oplus 64J_2 \oplus 14J_3$. Then, the hypothesis of Theorem \ref{maint} is satisfied, and we have:

\begin{cor}
  The semisimplification of $\ee$ as an object in $\Rep \bal_3$ under the adjoint action of $e_1 + e_2$ is $\g(6, 6)$.
\end{cor}

\subsubsection{Constructing $\g(8, 6)$}
In this section, we will construct the Elduque and Cunha Lie superalgebra $\g(8, 6)$. This is a simple contragredient Lie superalgebra with a Cartan matrix and parity set:
\begin{align*}
  \widetilde{A} = \begin{pmatrix}
  2 & 0 & -1 & 0 & 0 & 0 & 0\\
  0 & 0 & -1 & 0 & 0 & 0 & 0\\
  -1 & -1 & 2 & -1 & 0 & 0 & 0\\
  0 & 0 & -1 & 2 & -1 & 0 & 0\\
  0 & 0 & 0 & -1 & 2 & -1 & 0 \\
  0 & 0 & 0 & 0 & -1 & 2 & -1 \\
  0 & 0 & 0 & 0 & 0 & -1 & 2
  \end{pmatrix}; \ \ \ I = \{0,1,0,0,0,0,0\}.
\end{align*}
This is a Lie superalgebra of superdimension $(133|56)$. Its even subalgebra is $\ese$, and its odd part is the $56$-dimensional simple Weyl module $L(\omega_1)$. By comparing Cartan matrices, we realize $\ee$ as an object in $\Rep \bal_3$ with respect to $e_1$, where it decomposes as $\ee = 133J_1 \oplus 56J_2 \oplus J_3$. Then, the hypothesis of Theorem \ref{maint} is satisfied, and we have:

\begin{cor}
  The semisimplification of $\ee$ as an object in $\Rep \bal_3$ under the adjoint action of $e_1$ is $\g(8, 6)$.
\end{cor}

\subsubsection{Constructing $\g(3, 6)$}\label{altconst}
In this section, we will construct the Elduque and Cunha Lie superalgebra $\g(3, 6)$. This is a simple contragredient Lie superalgebra with a Cartan matrix and parity set:

\begin{align*}
  \widetilde{A} = \begin{pmatrix}
  0 & -1 & 0 & 0 \\
  -1 & 0 & -1 & 0 \\
  0 & -1 & 0 & -2 \\
  0 & 0 & - 1 & 2
  \end{pmatrix}; \ \ \ I = \{1,1,1,0\}.
\end{align*}
This is a Lie superalgebra of superdimension $(36|40)$. Its even subalgebra is $\mathfrak{sp}_8$, and its odd part is the $40$-dimensional simple module $L(\omega_3)'$. This module is the quotient of the $48$-dimensional Weyl module $L(\omega_3)$ by the $8$-dimensional tautological module over $\mathfrak{sp}_8$. We can construct this Lie superalgebra from $\ee$, but this will slightly differ from the main approach above. Let $x = e_1 + e_2 + e_6 + e_8$. Then, it is easily checked that with respect to the adjoint action of $x$, $\ee$ is an object in $\Rep \bal_3$, where it decomposes as $\ee = 36J_1 + 40J_2 + 44J_3$. We will show that:

\begin{thm}
  The semisimplification of $\ee$ as an object in $\Rep \bal_3$ under the adjoint action of $e_1 + e_2 + e_6 + e_8$ is $\g(3, 6)$.
\end{thm}

\begin{proof}
  The main theorem does not go through because the node corresponding to $e_6$ on the Dynkin diagram is not a boundary node, so we will proceed manually using the language of \S\ref{explicit} and Proposition \ref{strucconsts}. However, the key point that the generators $e_1, e_2, e_6, e_8$ pairwise commute still holds. Let us first consider what happens to the positive generators. Because of the new situation, we will also need to consider root vectors that are not attached to simple roots. Let:

  \begin{enumerate}
    \item $e_9 \coloneqq [x, e_3] = [e_1,e_3]$,
    \item $e_{10} \coloneqq [x, e_4] = [e_2, e_4]$,
    \item $e_{13} \coloneqq -[x, e_5] = [e_5, e_6]$,
    \item $e_{14} \coloneqq [e_6, e_7]$,
    \item $e_{15} \coloneqq [-e_8, e_7]$,
    \item $e_{22} \coloneqq [e_8, [e_6, e_7]]$, which also happens to equal $-[x, e_{14} - e_{15}]$.
  \end{enumerate}
   and similarly for $f$; these new indices are chosen in accordance with labeling in the SuperLie software and do not have any other meaning. In particular, note that $e_{14} - e_{15} = [x, e_7]$.
   \par
   Let's consider the upper triangular subalgebra, completed to a direct summand in $\Rep \bal_3$ (by symmetry, this will tell us what happens to the lower triangular subalgebra). We have the following copies of $J_3$: 
   
   \begin{enumerate}
     \item $f_1 \mapsto h_1 \mapsto e_1$;
     \item $f_2 \mapsto h_2 \mapsto e_2$;
     \item $f_6 \mapsto h_6 \mapsto e_6$;
     \item $f_8 \mapsto h_8 \mapsto e_8$;
     \item $e_7 \mapsto e_{14} - e_{15} \mapsto -e_{22}$.
   \end{enumerate}
  These will vanish in the semisimplification, and in particular the information attached to the generator $e_7$ is annihilated. The remaining generators yield odd generators, as we have the following copies of $J_2$: $e_3 \mapsto e_9$, $e_4 \mapsto e_{10}$, and $e_5 \mapsto -e_{13}$.
   \par
   In the semisimplification, the images of these odd generators do not generate the image (in fact, it is easily seen that along with their $f$ counterparts they generate the derived algebra of the rank $3$ Lie superalgebra $\g(2, 3)$ above). However, $\g(3, 6)$ is of rank $4$ and contains $\g(2, 3)^{(1)}$ as a subalgebra, which we know by just looking at their Cartan matrices, so we should be able to find another indecomposable whose image will serve as our final generator (it should span a copy of $J_1$ which will be a direct summand). In fact, the vector that will serve as our generator is $e_{14} + e_{15}$. The vectors  $e_{14} + e_{15}$, $-f_{14} - f_{15}$, and $h_6 - h_7 + h_8$ form an $\sll$-triple (in characteristic $3$). The element $e_{14} + e_{15}$ should be treated as a positive generator. Hence, in the semisimplification, we define the following vectors, where the LHS is the vector and the RHS is the subspace it spans:

   \begin{align*}
     \widetilde{e}_1 \coloneqq \ov{e_3 \mapsto e_9}; \ \ \ \widetilde{e}_2 \coloneqq \ov{e_4 \mapsto e_{10}}; \ \ \ \widetilde{e}_3 \coloneqq \ov{e_5 \mapsto -e_{13}}; \ \ \ \widetilde{e}_4 \coloneqq \ov{e_{14} + e_{15}} \\
     \widetilde{f}_1 \coloneqq \ov{f_9 \mapsto f_3}; \ \ \ \widetilde{f}_2 \coloneqq \ov{f_{10} \mapsto f_{4}}; \ \ \ \widetilde{f}_3 \coloneqq \ov{-f_{13} \mapsto f_{5}}; \ \ \ \widetilde{f}_4 \coloneqq \ov{-f_{14} - e_{15}} \\
     \widetilde{h}_1 \coloneqq \ov{h_3}; \ \ \ \widetilde{h}_2 \coloneqq \ov{h_4}; \ \ \ \widetilde{h}_3 \coloneqq \ov{h_5}; \ \ \ \widetilde{h}_4 \coloneqq \ov{h_6 - h_7 + h_8}.
   \end{align*}
   Now, using SuperLie, we can check that the collection of vectors $\{\widetilde{e}_i, \widetilde{f}_i, \widetilde{h}_i\}$ do indeed generate the semisimplification. Furthermore, direct computations show that they satisfy the same relations as their counterparts in $\g(3, 6)$, so this means we have a morphism from $\g(3, 6)$ to the semisimplification. Then, by comparing dimensions, we deduce that they are isomorphic. As a remark, we do not need to even know that the generators generate the image, as the Lie superalgebra  $\g(3, 6)$ is simple.
\end{proof}

For completeness, we briefly describe another construction of $\g(3, 6)$. Recall that as a pair (Lie algebra, module), $\ee$ splits as the $120$-dimensional simple Lie algebra $\mathfrak{o}_{16}$ and its $128$-dimensional spinor module $L(\omega_8)$. In our notation, the subalgebra $\mathfrak{o}_{16}$ is generated by $e_2, e_3, e_4, e_5$, $e_6, e_7, e_8$ and

\[e_{100} \coloneqq [[[[e_1,e_3], [e_4,e_5]], [[e_2,e_4],[e_5,e_6]]], [[[e_1,e_3],[e_2,e_4]], [[e_6,e_7], [e_5,[e_3,e_4]]]]],\]
and their $f$ counterparts (again, the choice of index here is just based on the program SuperLie and otherwise has no meaning).
\par
Let $e_{27} \coloneqq [[e_6,e_5],[e_4,e_3]]$ and similarly for $f_{27}$. Then, $\mathfrak{sp}_8$ is a subalgebra of $\mathfrak{o}_{16}$ generated by $e_{100} - f_{5}$, $e_{8} - f_{4}$, $e_7 - f_3$, $e_{27}$, and  $f_{100} - e_{5}$, $f_{8} - e_{4}$, $f_7 - e_3$, $f_{27}$, where the first four elements correspond to the simple roots in the usual order and the last four elements correspond to the negatives of the simple roots. Then, one can take an $\sll$-triple in $\mathfrak{o}_{16} \subseteq \ee$ that centralizes $\mathfrak{sp}_8$; semisimplifying with respect to, say, the positive root vector in this triple gives the desired $\g(3, 6)$.

\subsubsection{Summary}
In this section, we summarize and extend the results above. Here we have the appropriate Dynkin diagrams for easy reference, where the diagrams for $\esi$ and $\ese$ are subdiagrams of that of $\ee$ in the obvious way:

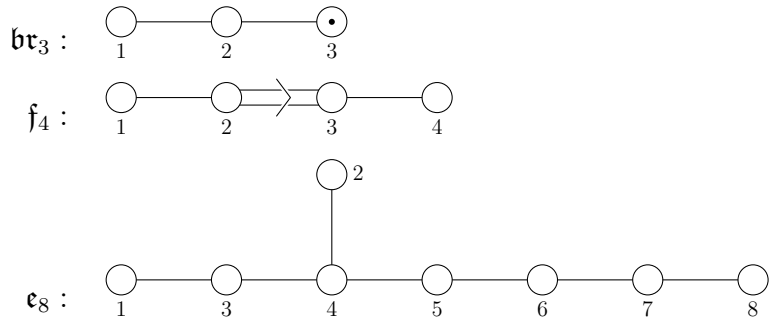
\begin{center}
  \begin{align*}
    \mathfrak{br}_3: \hspace{5 mm} & 
    \begin{dynkinDiagram}[labels={1,2,3},scale=4, mark=o, label distance = 5pt]A3  
    \fill[black] (root 3) circle (0.01cm);
    \end{dynkinDiagram}\\    
    \ff: \hspace{5 mm} &  \dynkin[labels={1,2,3,4},scale=4, arrow shape/.style={-{jibladze[length=7pt]}}, label distance = 5pt]  F{oooo} \\
    \ee: \hspace{5 mm}&  \dynkin[labels={1,2,3,4,5,6,7,8},scale=4, label distance = 5pt]  E{oooooooo}
  \end{align*}
\end{center}
In the table below, we state results by specifying the starting Lie algebra, the nilpotent element used to semisimplify, and the resulting Lie superalgebra. We include all possible combinations of boundary nodes and some examples that do not follow from the main theorem as well. Other nilpotent elements (but not all of them) that give the same semisimplification are listed in the same row. For the constructions involving $\mathfrak{e}_6^{(1)}$, one can further quotient out by the centers to get simple Lie (super)algebras.

\begin{longtable}[c]{@{}lll@{}}
\toprule
 Lie algebra & Nilpotent element  & Lie superalgebra \\* \midrule
\endfirsthead
\multicolumn{3}{c}%
{{\bfseries Table continued from previous page}} \\
\toprule
 Lie algebra & Nilpotent element(s)  & Lie superalgebra  \\* \midrule
\endhead
\bottomrule
\endfoot
\endlastfoot
$\mathfrak{br}_3$ & $e_1, e_2$ & $\mathfrak{brj}_{2;3}$  \\ \midrule
 $\ff$ & $e_1$ & see ($\star$) below  \\
 &  $e_4$ & $\g(1,6)$  \\
 & $e_1 + e_4$ & see ($\star$) below  \\ \midrule
 $\esi^{(1)}$ & $e_1, e_2, e_6$ &  $\g(2, 6)^{(1)}$ \\
 & $e_1 + e_2$, $e_2 + e_6$, $e_1 + e_6$  & $\g(3,3)^{(1)}$  \\
 & $e_1 + e_2 + e_6$  & $\g(2,3)^{(1)}$ \\ \midrule
 $\ese$ & $e_1, e_2, e_7$ & $\g(4,6)$  \\
 & $e_1 + e_2$, $e_2 + e_7$, $e_1 + e_7$ &  $\mathfrak{el}(5; 3)$ \\
 & $e_1 + e_2 + e_7$  & $\g(4,3)$ \\
 & $e_2 + e_5 + e_7$ & $\ff$;  see  $(\star\star)$  below \\
 & $e_1 + e_2 + e_5 + e_7$  & $\g(1, 6)$ \\ \midrule
 $\ee$ & $e_1, e_2, e_8$ &  $\g(8, 6)$  \\
 & $e_1 + e_2$, $e_2 + e_8$, $e_1 + e_8$  & $\g(6, 6)$ \\
 & $e_1 + e_2 + e_8$  &  $\g(8, 3)$  \\
 & $e_1 + e_2 + e_6 + e_8$ & $\g(3, 6)$  \\* \bottomrule
\end{longtable}
\begin{enumerate}
  \item[($\star$)] Semisimplifying $\ff$ with Cartan matrix $A$ with respect to $e_1$ gives a Lie superalgebra $\ov{\ff}$ of superdimension of $(15|8)$, whereas the Lie superalgebra $\g(\widetilde{A}) = \mathfrak{sl}_{3|1}$ as described in the setup for Theorem \ref{maint} is of superdimension $(9|6)$. Hence, Theorem \ref{maint} cannot be applied. However, it can be computed by hand that $\ov{\ff}^{gen}$ is of superdimension $(9|6)$ and is isomorphic to $\mathfrak{sl}_{3|1}$. A similar statement applies when we semisimplify $\ff$ with $e_1 + e_4$. Although these may seem like edge cases, there is a more general conjecture which captures these cases and the main theorem simultaneously, which we discuss in \S\ref{remarks} in Conjecture \ref{conj}.
  \item[($\star\star$)] The semsimplification of $\ese$ giving $\ff$ is related to the Kantor-Koecher-Tits construction of a Lie algebra given a Jordan algebra. In particular, $\ese = \sll \otimes \mathbb{A} \oplus \ff$ as vector spaces, where $\mathbb{A}$ is the Albert algebra, which is the exceptional simple Jordan algebra. Furthermore, there appears to be some notion of ``iterating" semisimplifications, as we can semisimplify $\ese$ to get $\ff$ and then semisimplify again to get $\g(1,6)$, or we can semisimplify $\ese$ to directly get $\g(1,6)$.
\end{enumerate}
\par
One may notice that for $\esi, \ese, \ee$, when we picked fewer than three generators to sum, it did not matter which generators we picked to semisimplify, but rather only how many we picked. These are actually instances of a more general phenomenon. Let $\g = \g(A)$ be as in the hypothesis of Theorem \ref{maint}. which is found at the start of \S\ref{maintsec}. Let $\mathcal{I}$ be any subset of the nodes of the Dynkin diagram $\g$ such that no two nodes in $\mathcal{I}$ are adjacent. Color the nodes in $\mathcal{I}$ black, and all other nodes white. Finally, let $e_\mathcal{I} \coloneqq \sum_{i \in \mathcal{I}} e_i$.
\par
A \textit{legal swap} of the colored Dynkin diagram is defined as any recoloring of the diagram of the following form: if $i$ and $j$ are two adjacent nodes connected by a single edge, with $i$ colored black, and $j$ colored white, and no other node adjacent to $j$ is black, swap the colors of $i$ and $j$. It is clear that the legal swaps generate a groupoid under composition; we will call its elements \textit{legal recolorings}. Let $\sigma(\mathcal{I})$ denote the set of nodes colored black after applying the legal recoloring $\sigma$.
\par
Here are some examples of legal swaps:

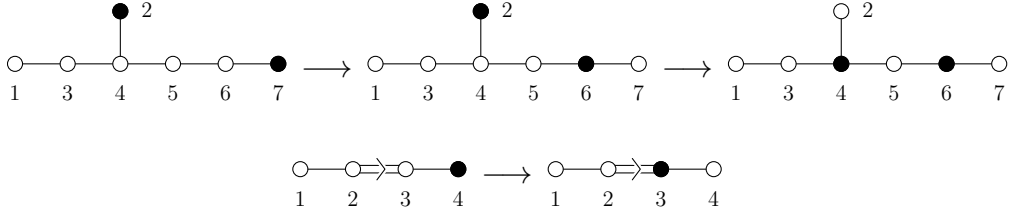
\begin{center}
  \vspace{0.5 cm}
  \begin{dynkinDiagram}[mark=o,labels={1,2,3,4,5,6,7},scale=2, label distance = 5pt]E7
  \fill[black, draw=black] (root 2) circle (0.05cm);
  \fill[black, draw=black] (root 7) circle (0.05cm);
  \end{dynkinDiagram} $\longrightarrow$
  \begin{dynkinDiagram}[mark=o,labels={1,2,3,4,5,6,7},scale=2, label distance = 5pt]E7
  \fill[black, draw=black] (root 2) circle (0.05cm);
  \fill[black, draw=black] (root 6) circle (0.05cm);
  \end{dynkinDiagram} $\longrightarrow$
  \begin{dynkinDiagram}[mark=o,labels={1,2,3,4,5,6,7},scale=2, label distance = 5pt]E7
  \fill[black, draw=black] (root 4) circle (0.05cm);
  \fill[black, draw=black] (root 6) circle (0.05cm);
  \end{dynkinDiagram}
\end{center}

\begin{center}
  \vspace{0.5 cm}
  \begin{dynkinDiagram}[mark=o,labels={1,2,3,4},scale=2, arrow shape/.style={-{jibladze[scale=0.58, length=7pt]}}, label distance = 5pt]F4
  \fill[black, draw=black] (root 4) circle (0.05cm);
  \end{dynkinDiagram} $\longrightarrow$
  \begin{dynkinDiagram}[mark=o,labels={1,2,3,4},scale=2, arrow shape/.style={-{jibladze[scale=0.58, length=7pt]}}, label distance = 5pt]F4
  \fill[black, draw=black] (root 3) circle (0.05cm);
  \end{dynkinDiagram}
\end{center}
Here are some examples of illegal swaps:
\begin{center}
  \vspace{0.5 cm}
  \begin{dynkinDiagram}[mark=o,labels={1,2,3,4,5,6,7},scale=2, label distance = 5pt]E7
  \fill[black, draw=black] (root 4) circle (0.05cm);
  \fill[black, draw=black] (root 6) circle (0.05cm);
  \end{dynkinDiagram} $\longrightarrow$
  \begin{dynkinDiagram}[mark=o,labels={1,2,3,4,5,6,7},scale=2, label distance = 5pt]E7
  \fill[black, draw=black] (root 4) circle (0.05cm);
  \fill[black, draw=black] (root 5) circle (0.05cm);
  \end{dynkinDiagram}
\end{center}

\begin{center}
  \vspace{0.5 cm}
  \begin{dynkinDiagram}[mark=o,labels={1,2,3,4},scale=2, arrow shape/.style={-{jibladze[scale=0.58, length=7pt]}}, label distance = 5pt]F4
  \fill[black, draw=black] (root 3) circle (0.05cm);F
  \end{dynkinDiagram} $\longrightarrow$
  \begin{dynkinDiagram}[mark=o,labels={1,2,3,4},scale=2, arrow shape/.style={-{jibladze[scale=0.58, length=7pt]}}, label distance = 5pt]F4
  \fill[black, draw=black] (root 2) circle (0.05cm);
  \end{dynkinDiagram}
\end{center}
\begin{thm}\label{swaps}
  Let $\mathcal{I}$ be a configuration of black nodes as above, and let $\sigma$ be a legal recoloring. Then, $\g$ can be realized as an object in $\Rep \bal_3$ in two ways: with respect to $e_\mathcal{I}$, denoted $\g_\mathcal{I}$, or with respect to $e_{\sigma(\mathcal{I})}$, denoted $\g_{\sigma(\mathcal{I})}$. Furthermore, the semisimplifications $\ov{\g_\mathcal{I}}$  and $\ov{\g_{\sigma(\mathcal{I})}}$ are isomorphic as Lie superalgebras.
\end{thm}

\begin{proof}
  It suffices to assume that $\sigma$ is a legal swap that swaps the colors of the adjacent nodes $i$ and $j$, where $i$ is black and $j$ is white. By the initial configuration $\mathcal{I}$ and the definition of a legal swap, the proof that $e_I$ is nilpotent of degree at most three goes through just like the proof of Lemma \ref{obincat}.
  \par
  Now, let $G$ be a split simple linear algebraic group with Lie algebra $\g$. Clearly, the adjoint representation of $G$ acts on $\g$ by Lie algebra automorphisms, and nilpotent elements are partitioned into orbits. Hence, it suffices to show that $e_\mathcal{I}$ and $e_{\sigma(\mathcal{I})}$ lie in the same nilpotent orbit.
  \par
  The Dynkin subdiagram formed by nodes $i$ and $j$ corresponds to a subgroup $H$ of type $A_2$ (i.e. $SL_3$ or $PGL_3$) in $G$, and there is a Weyl group element $w \in N(T)/T$, where $T$ is a maximal torus in $H$, that permutes them. We can find a coset representative $g \in H$ that corresponds to $w$ such that its conjugation action on $\mathfrak{g}$ will send $e_i$ to $e_j$. The element $g$ then lifts to an element $\widetilde{g}$ in $G$, and because no other black nodes are connected to $i$ or $j$, the conjugation action of $\widetilde{g}$ will not change any other $e_k$ for $k \neq i \in \mathcal{I}$. This shows the claim.
\end{proof}
Combining this theorem with Theorem \ref{maint} gives us the semisimplification with respect to a large class of elements. For instance, we deduce that semisimplifying $\ee$ with respect to any $e_i$, regardless of whether $i$ corresponds to a boundary node or not, gives $\g(8, 6)$.

\subsection{Some Remarks} \label{remarks}
\subsubsection{Infinite-Dimensional Lie Algebras}
Although we worked with finite-dimensional Lie algebras, it is natural to see what Lie superalgebras we get when we allow infinite-dimensional Lie algebras.
\par
Let $A$ be a purely even, symmetrizable, indecomposable Cartan matrix of size $n$ whose off-diagonal entries can take values in $\{-2, -1, 0\}$. Let $D$ be the corresponding Dynkin diagram, and call a node in the diagram a \textit{boundary node} if it is connected to exactly one other node in the diagram via a single edge. Let $\mathcal{I} = \{i_1, \dots, i_l\}$ be a subset of the boundary nodes of $D$ such that no two nodes in $\mathcal{I}$ share an adjacent node, and let $\mathcal{J} = \{j_1, \dots, j_l\}$ be the corresponding adjacent nodes, respectively (the elements of $\mathcal{J}$ are necessarily distinct). Furthermore, we require that $\mathcal{I}$ be chosen so that $a_{ii} = 2$ for all $i \in \mathcal{I} \cup \mathcal{J}$. Let $\widetilde{A}$ be the $(n-l) \times (n-l)$ matrix obtained from $A$ by setting $a_{jj} = 0$ for $j \in \mathcal{J}$ and deleting the $i$-th row and column from $A$ for $i \in \mathcal{I}$. Finally, let $e_\mathcal{I} = \sum_{i\in \mathcal{I}} e_i$. Then, we have the following conjecture for characteristic $3$:

\begin{conj}\label{conj}
  The Lie algebra $\g(A)^{(1)}$ can be realized as an object in $\Rep \bal_3$ with respect to $e_\mathcal{I}$, and $\ov{\g(A)^{(1)}}^{gen}$ is isomorphic to $\g(\widetilde{A})^{(1)}$. Furthermore, if $\sigma$ is any legal recoloring of $\mathcal{I}$, then the previous statement also holds for $e_{\sigma(\mathcal{I})}$ in place of $e_\mathcal{I}$.
\end{conj}
Most of the work for the first statement of the conjecture has already been done when proving Theorem \ref{maint}; in fact, if we assume that if for each $n \in \Z$ that each subspace of height $n$ in $\g(\widetilde{A})^{(1)}$ and in $\ov{\g(A)^{(1)}}^{gen}$ have the same dimension, then the proof actually goes through directly. However, loosening this dimension requirement would be interesting because it does not appear to be necessary in our computations. In any case, there should be a surjection from $\ov{\g(A)^{(1)}}^{gen}$ to $\g(\widetilde{A})^{(1)}$. The difficulty lies in showing that the former has no nontrivial ideal that intersects the Cartan subalgebra. The last statement of the conjecture should follow from the proof of Theorem \ref{swaps}.
\par
It would be interesting to develop an $\sll$-equivariant theory of infinite-dimensional Lie algebras with Cartan matrix, or more generally, a theory of infinite-dimensional Lie algebras with Cartan matrix in the category of $\Rep \bal_p$; then one can see what happens in the semisimplification in $\Ver_p$ and in $\sVec_\KK$. 

\subsubsection{Applications to Representation Theory}
Lastly, one can use the representation theory of exceptional Lie algebras to study representation theory of the exceptional Lie superalgebras obtained via semisimplification of these exceptional Lie algebras. In particular, one can start with a module over an exceptional Lie algebra in $\Rep \bal_3$ and semisimplify it to construct modules over the corresponding Lie superalgebra. Although these can be probably classified by a highest-weight argument, this approach will give a construction to determine the size of these modules, which to our knowledge is virtually unknown.
\section{An Exceptional Lie Superalgebra in Characteristic $5$}

In this section, we construct the Elduque Lie superalgebra $\mathfrak{el}(5; 5)$ of superdimension $(55|32)$ by semisimplifying $\ee$. This Lie superalgebra appeared for the first time in \cite{elduque2007some}. The even subalgebra is the orthogonal Lie algebra $\mathfrak{o}_{11}$, and the odd part is the $32$-dimensional spinor module $L(\omega_5)$. The Lie superalgebra $\mathfrak{el}(5; 3)$ constructed earlier is a characteristic $3$ analog of this discovered in \cite{bouarroudj2009classification}, in the sense that there is a suitable choice of Cartan matrix that creates both, depending on the characteristic. Our main theorem has not been extended to characteristic $5$, so we will construct this in a fashion similar to the alternate construction of $\g(3, 6)$ described above.
\par
As remarked earlier, $\ee$ as a vector space is the direct sum of the $120$-dimensional simple Lie algebra $\mathfrak{o}_{16}$ and its $128$-dimensional spinor module $L(\omega_8)$. First, let's consider the subalgebra $O = \mathfrak{o}_{16}$. Recall from \S\ref{altconst} that the root vectors $e_2, e_3, e_4, e_5$, $e_6, e_7, e_8$ and $e_{100}$, which along with their $f$ counterparts generate $O$. Visually, the Dynkin diagram looks like:

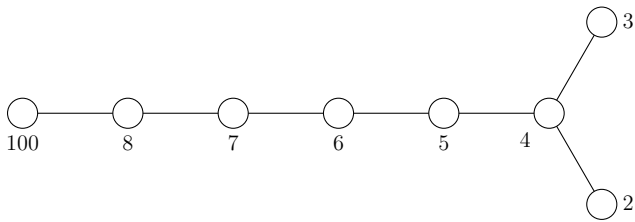
\[\dynkin[labels={100,8,7,6,5,4,3,2},scale=4, label distance = 5pt]  D{oooooooo}\]
Let $x = e_2 + e_3 + e_4$. First of all, because  $(\ad x)^5 = 0$, we can realize $\ee$ as an object in $\Rep \bal_5$ with respect to $\ad x$. Furthermore, because $x \in O$ and $O$ is a subalgebra, we can also view $O$ as a subobject of $\ee$ in $\Rep \bal_5$. The Lie algebra $O$ also acts on $V = \KK^{16}$; let us view $V$ as an object in $\Rep \bal_5$ with respect to the action by $x$. Then, the action $O \otimes V \rightarrow V$ becomes a morphism in $\Rep \bal_5$ as well, and $O$ is precisely the subalgebra of $\gl(V)$ preserving some non-degenerate symmetric bilinear form $\gamma$ on $V$.
\par
Then, it is easily checked that $V = 11J_1 \oplus J_5$, such that the $11J_1$ and the $J_5$ are orthogonal to each other, and the restriction of the form to each piece is non-degenerate. One way to see this is to note that $e_2 + e_3, e_4, f_2 + f_3, f_4$ generate $\mathfrak{o}_5 = \mathfrak{sp}_4$ as a subalgebra, and $e_2 + e_3 + e_4$ acts as a Jordan block of size $4$ on the four-dimensional module over $\mathfrak{sp}_4$ and hence as a Jordan block of size $5$ on the five-dimensional module over $\mathfrak{o}_5$ (and once embedded in $\mathfrak{o}_{16}$, it acts trivially on its orthogonal complement, which is $11$-dimensional). Then, when we semisimplify, $O = \mathfrak{o}_{16}$ becomes $\ov{O} = \mathfrak{o}_{11}$ by the arguments in \S\ref{ssclassical}.
\par
Now, let's check what happens to $L(\omega_8)$, which we view as a subspace of $\ee$. The direct sum $\ee = \mathfrak{o}_{16} \oplus L(\omega_8)$ is a direct sum in $\Rep \bal_5$ as well. It can be checked using the software SuperLie that $L(\omega_8) = 32J_4$ in $\Rep \bal_5$, so its semisimplification is purely odd of dimension $32$. We claim that the semisimplification is actually the spinor module $L(\omega_5)$ over $\mathfrak{o}_{11}$.
\par
Since $V = 11J_1 \oplus J_5$ is an orthogonal decomposition in $\Rep \bal_5$, the action of $\mathfrak{o}_{16}$ restricts to an action of $\mathfrak{o}_{11} \oplus \mathfrak{o}_{5}$ on the spinor module $L(\omega_8)$. It is well-known that as an $\mathfrak{o}_{11} \oplus \mathfrak{o}_{5}$-module $L(\omega_8) = L(\omega_5) \otimes L(\omega_2)$ is the tensor product of the corresponding spinor modules. In a suitable basis, the module $L(\omega_5)$ is $32J_1$, and the module $L(\omega_2)$ is $J_4$. After semisimplification, the claim follows.
\par
Therefore, we deduce that the semisimplification of $\ee$ with respect to the adjoint action of $x$ decomposes as $55L_1 \oplus 32L_4$ in $\Ver_5$, so it can be identified with a Lie superalgebra. The even part is the simple Lie algebra $\mathfrak{o}_{11}$, and the odd part is its simple module $L(\omega_5)$. Therefore, by the classification in \cite{bouarroudj2009classification}, this must be the Lie superalgebra $\mathfrak{el}(5;5)$. We have proved the following theorem:

\begin{thm2}
  The semisimplification of $\ee$ as an object in $\Rep \bal_5$ under the adjoint action of $e_2 + e_3 + e_4$ is a Lie algebra in $\Ver_5$ of the form $55L_1 \oplus 32L_4$; in particular, this is the Elduque Lie superalgebra $\mathfrak{el}(5;5)$.
\end{thm2}

\printbibliography

\end{document}